\documentclass[reqno,12pt]{amsart}
\usepackage{a4wide}

\usepackage{varioref}   %
\usepackage{hyperref}
\usepackage{cleveref}

\usepackage{amsmath,amssymb,amsthm,amsfonts}
\usepackage{mathrsfs}
\usepackage{bbm}
\usepackage{hyperref}

\usepackage{color}

\newtheorem{lemma}{Lemma}[section]
\newtheorem{theorem}{Theorem}[section]

\newtheorem{proposition}{Proposition}[section]

\numberwithin{equation}{section}

\newcommand{\dis}{\displaystyle}

\newcommand{\CE}{\mathcal{E}}

\newcommand{\CH}{\mathcal{H}}

\newcommand{\CL}{\mathcal{L}}
\newcommand{\CN}{\mathcal{N}}

\newcommand{\CD}{\mathcal{D}}
\newcommand{\CR}{\mathcal{R}}

\newcommand{\la}{\lambda}

\newcommand{\si}{\sigma}

\newcommand{\Ga}{\Gamma}

\begin{document}

\title[A half-space problem on the full Euler-Poisson system]{A half-space problem on the full Euler-Poisson system}

\author[R.-J. Duan]{Renjun Duan}
\address[RJD]{Department of Mathematics, The Chinese University of Hong Kong,
Shatin, Hong Kong}
\email{rjduan@math.cuhk.edu.hk}

\author[H.-Y. Yin]{Haiyan Yin}
\address[HYY]{School of Mathematical Sciences, Huaqiao University, Quanzhou 362021, P.R.~China.}
\email{yinhaiyan2000@aliyun.com}

\author[C.-J. Zhu]{Changjiang Zhu}
\address[CJZ]{School of Mathematics, South China University of Technology, Guangzhou, 510641, P.R.~China}
\email{machjzhu@scut.edu.cn}

\keywords{Stationary solution, asymptotic
stability, convergence rate, weighted energy method.\\
\mbox{\quad Corresponding author: Haiyan Yin}}


\date{\today}
\maketitle

\begin{abstract}
This paper is concerned with the initial-boundary value problem on the full Euler-Poisson system for ions over a half line. We establish the existence of stationary solutions under the Bohm criterion similar to the isentropic case and further obtain the large time asymptotic stability of small-amplitude stationary solutions provided that the initial perturbation is sufficiently small in some weighted Sobolev spaces. Moreover, the convergence rate of the solution toward the stationary solution is obtained. The proof is based on the energy method. A key point is to capture the positivity of the temporal energy dissipation functional and boundary terms with suitable space weight functions either algebraic or exponential depending on whether or not the incoming far-field velocity is critical.
\end{abstract}

%

\thispagestyle{empty}

\section{Introduction}
In plasmas confined in half space by a wall, the self-consistent potential may induce a sheath near the wall so as to realize the balance between the reflected electrons and the leaving ions, cf.~\cite{FFC,Riemann}. Mathematically, the plasma sheath is often described as the stationary solution or boundary layer solution in the half line to the Euler-Poisson system for the only heavier ions flow under the Boltzmann relation. The existence of the plasma sheath is guaranteed by the Bohm criterion saying that ions must move toward the wall at infinity with a velocity greater than a critical value given particularly as the acoustic velocity for cold ions. In fact, the Bohm criterion is also a condition on the large time asymptotic stability of the stationary solutions under consideration. The relevant mathematical studies has been done in \cite{HS,NOS,boundaawa,Suzuki,multicomponent}.

In the paper, we will take into account the additional effect of the variable temperature. For this purpose, the flow of positively charged ions in plasmas is governed by the full Euler-Poisson system of the form
\begin{eqnarray}\label{1.1}
&&\left\{\begin{aligned}
& n_t+(nu)_{x}=0,\\
& (m n u)_t+(mnu^{2}+p)_x
=n\phi_x,\\
&W_{t}+(Wu+pu)_{x}=nu\phi_x,\\
&\phi_{xx}=n-e^{-\phi}.
\end{aligned}\right.
\end{eqnarray}
The unknown functions $n$, $u$ and  $\phi$  stand for the density, velocity and the electrostatic potential, respectively. The positive constant $m$ is the mass of an ion. The function $W$ stands for the total
energy given by
\begin{eqnarray}\label{1.2}
W=\frac{1}{2}mnu^{2}+\frac{p}{\gamma-1},
\end{eqnarray}
where the constant $\gamma>1$ is the ratio of specific heats and the pressure $p$ satisfies the equation of state:
\begin{eqnarray}\label{1.3}
p=RTn,
\end{eqnarray}
with the temperature function $T$ and the Boltzmann constant $R>0$.
Note that $\phi$ has been chosen to have an opposite sign compared to the usual situation in physics. In the fourth equation of \eqref{1.1}, the electron density $n_{e}$ is determined by the  electrostatic potential in terms of the Boltzmann relation $n_{e}=e^{-\phi}$. We remark that the full Euler-Poisson system \eqref{1.1} can be formally derived through the macro-micro decomposition from the Vlasov-Poisson-Boltzmann system for the ions flow in kinetic theory, cf.~\cite{DL}.
 Substituting $\eqref{1.2}$ and $\eqref{1.3}$ into the equations
 $\eqref{1.1}$, we can obtain the following system of equations:
\begin{eqnarray}\label{1.4}
&&\left\{\begin{aligned}
& n_t+(nu)_{x}=0,\\
& m n (u_t+uu_{x})+(RTn)_x
=n\phi_x,\\
&T_{t}+uT_{x}+(\gamma-1)Tu_{x}=0,\\
&\phi_{xx}=n-e^{-\phi}.
\end{aligned}\right.
\end{eqnarray}

The goal of this paper is to study the long-time behavior of solutions to the initial boundary value problem on $\eqref{1.4}$ over the one-dimensional half space $\mathbb{R}_{+}:=\{x>0\}$.
Initial data for system $\eqref{1.4}$ are given by
\begin{eqnarray}\label{1.5}
(n,u,T)(0,x)=(n_{0},u_{0},T_{0})(x)\quad \text{with}\ \
\inf_{x\in \mathbb{R}_{+}}n_{0}(x)>0\ \ \text{and} \ \ \inf_{x\in \mathbb{R}_{+}}T_{0}(x)>0.
\end{eqnarray}
We assume that initial data at the far-field $x=\infty$ are
constant, namely,
\begin{eqnarray}\label{1.5a}
\lim_{x\rightarrow\infty}(n_{0},u_{0},T_{0})(x)=(n_{\infty},u_{\infty},T_{\infty}),\
\ \ n_{\infty}>0, \ \ \  T_{\infty}>0.
\end{eqnarray}
Since the fourth equation of $\eqref{1.4}$ is elliptic, the
extra boundary data on $\phi$ has to be supplemented as
\begin{eqnarray}\label{1.5b}
\phi(t,0)=\phi_{b}\neq 0,\ \ \ \ \ \ \lim_{x\rightarrow\infty}\phi(t,x)=0.
\end{eqnarray}
To the end, we always assume that
\begin{eqnarray}\label{1.6}
n_{\infty}=1,
\end{eqnarray}
so that the quasi-neutrality  holds true at $x=\infty$ by $\eqref{1.5b}$ and $\eqref{1.6}$. In particular, in this paper, we are only concerned with the existence and asymptotic stability of stationary solutions to the initial boundary value problem \eqref{1.4}, \eqref{1.5}, \eqref{1.5a}, \eqref{1.5b} and \eqref{1.6}. For this purpose, we denote $(\tilde n, \tilde
u, \tilde T, \tilde\phi)$ to be the solution of the stationary problem on the half space
\begin{eqnarray}\label{1.7}
&&\left\{\begin{aligned}
&(\tilde n\tilde u)_{x}=0,\\
& m \tilde n \tilde u \tilde u_x +(R \tilde T\tilde n)_{x}=\tilde n \tilde{\phi}_x,\\
&\tilde u \tilde T_{x}+(\gamma-1)\tilde T\tilde u_x=0,\\
 &\tilde{\phi}_{xx}=\tilde n-e^{-\tilde{\phi}}.
\end{aligned}\right.
\end{eqnarray}
Here, corresponding to $\eqref{1.5}$--$\eqref{1.6}$, we also require that \eqref{1.7} is supplemented with
\begin{equation}\label{1.8}
\left\{\begin{aligned}
  & \inf_{x\in \mathbb{R}_{+}}\tilde n(x)>0,\ \ \
\inf_{x\in \mathbb{R}_{+}}\tilde T(x)>0,\\
&\lim_{x\rightarrow\infty}(\tilde n,\tilde u,\tilde T, \tilde
\phi)(x)=(1,u_{\infty},T_{\infty},0),\\
&\tilde\phi(0)=\phi_{b}.
\end{aligned}\right.
\end{equation}
It is convenient to call $(\tilde n, \tilde
u, \tilde T, \tilde\phi)$ the boundary layer solution. In case $\phi_{b}=0$, if uniqueness is assumed, then one can only get the trivial solution $(\tilde n, \tilde
u, \tilde T, \tilde\phi)=(1,u_{\infty}, T_{\infty}, 0)$.
 Thus we consider the boundary layer solution under the assumption that $\phi_{b}\neq 0$.

Notice that the first and third equations of \eqref{1.7} together with the boundary data \eqref{1.8} give that
\begin{equation*}
\tilde{T}=T_\infty \left(\frac{\tilde{n}}{n_\infty}\right)^{\gamma-1}=T_\infty \tilde{n}^{\gamma-1}.
\end{equation*}
Plugging it into the second equation of \eqref{1.7}, it can be reduced to the isentropic case that has been extensively studied in \cite{Suzuki}. In fact, to consider the existence of stationary solutions, the Sagdeev potential
\begin{equation}\label{1.9}
\left\{\begin{aligned}
& V(\phi):=\int_{0}^{\phi}
[f^{-1}(\eta)-e^{-\eta}]d\eta,\\
&\qquad\qquad\text{with }f(n)=\frac{\gamma
RT_{\infty}}{\gamma-1}\left(n^{\gamma-1}-1\right)+\frac{mu^{2}_{\infty}}{2 }\left(\frac{1}{n^{2}}-1\right)
\end{aligned}\right.
\end{equation}
plays a crucial role. One can compute that
\begin{eqnarray*}
&&\begin{aligned} &  f'(n)=\frac{-mu_{\infty}^{2}+\gamma
RT_{\infty}n^{\gamma+1}}{n^{3}}.
\end{aligned}
\end{eqnarray*}
Then, the only critical point of $f$ occurs at
\begin{equation*}
n=c_\infty:=\left(\frac{m
u_{\infty}^{2}}{\gamma RT_{\infty}}\right)^{\frac{1}{\gamma+1}},
\end{equation*}
where the constant $c_\infty$ is determined by the far-field data in connection with the Mach number at $x=\infty$. Therefore, in terms of the critical point $c_\infty$, the inverse function $f^{-1}$ in \eqref{1.9} is understood by adopting the branch which
contains the far-field equilibrium state $(\tilde{n}, \tilde{\phi})=(1, 0)$. Since the unique existence
of the monotone stationary solution can be proved by a method similar to that in
 \cite{NOS} and \cite{Suzuki}, we omit the detailed discussions for brevity and list the main results in the following

\begin{proposition}\label{prop1.1}
Consider the boundary-value problem \eqref{1.7} and \eqref{1.8}.
\begin{itemize}

\item[(i)] Let $u_{\infty}$ be a constant satisfying
$$
\text{either}\ \
u^{2}_{\infty}\leq \frac{\gamma RT_{\infty}}{m}
\ \ \text{or}\ \
\frac{\gamma
RT_{\infty}+1}{m}\leq u^{2}_{\infty}.
$$
Then the stationary problem
$\eqref{1.7}$ and $\eqref{1.8}$ has a unique monotone solution
$(\tilde n, \tilde u, \tilde T, \tilde\phi)$ verifying
\begin{eqnarray*}
&&\begin{aligned} &\tilde n,\tilde u, \tilde T, \tilde\phi \in
C(\overline{\mathbb{R}}_{+}),\ \ \ \tilde n, \tilde u, \tilde T,
\tilde \phi, \tilde\phi_{x} \in C^{1}(\mathbb{R}_{+})
\end{aligned}
\end{eqnarray*}
if and only if the boundary data $\phi_{b}$ satisfies conditions
\begin{eqnarray*}
&&V(\phi_{b})\geq 0,\ \ \  \phi_{b}\geq f(c_\infty).
\end{eqnarray*}

\item[(ii)]Let $u_{\infty}$ be a constant satisfying
$$
\frac{\gamma
RT_{\infty}}{m}<u^{2}_{\infty}<\frac{\gamma RT_{\infty}+1}{m}.
$$
If
$\phi_{b}\neq 0$, then the stationary problem $\eqref{1.7}$ and
$\eqref{1.8}$ does not admit any solutions in the function space
$ C^{1}(\mathbb{R}_{+}). $ If $\phi_{b}= 0$, then a constant
state $(\tilde n, \tilde u, \tilde T,
\tilde\phi)=(1,u_{\infty},T_{\infty},0)$ is the unique solution.
\end{itemize}
Moreover, the existing stationary solution enjoy some additional space-decay properties in the following two cases:
\begin{itemize}

\item (Nondegenerate case) Assume that
$$
\frac{\gamma RT_{\infty}+1}{m}<
u^{2}_{\infty},\quad {u_\infty<0,}
$$
and
$
\phi_{b}\neq f(c_\infty)
$
hold true. The
stationary solution $(\tilde n, \tilde u, \tilde T, \tilde\phi)$  belongs to
$C^{\infty}(\overline{\mathbb{R}}_{+})$ and verifies
\begin{eqnarray}\label{1.14}
&&\begin{aligned} &|\partial_{x}^{i}(\tilde
n-1)|+|\partial_{x}^{i}(\tilde
u-u_{\infty})|+|\partial_{x}^{i}(\tilde
T-T_{\infty})|+|\partial_{x}^{i}\tilde \phi|\leq C |
\phi_{b}|e^{-cx},
\end{aligned}
\end{eqnarray}
for any $i\geq 0$,
where $c$ and $C$ are positive constants.

\item (Degenerate case) Assume that
$$
\frac{\gamma RT_{\infty}+1}{m}= u^{2}_{\infty},\quad {u_\infty<0,}
$$
and
$
\phi_{b}>0
$
hold true. Denote constants
\begin{equation*}
\left\{\begin{aligned}
&c_{0}=1,\\
&c_{1}=-2\Gamma,\\
&c_{2}=\frac{(\gamma^{2}+\gamma)RT_{\infty}+2}{2},\\
&c_{3}=-2\Gamma[(\gamma^{2}+\gamma)RT_{\infty}+2],
\end{aligned}\right.
\end{equation*}
with
\begin{equation}
\label{def.conGa}
\Gamma=\sqrt{\frac{(\gamma^{2}+\gamma)RT_{\infty}+2}{12}}.
\end{equation}
There are constants $\delta_0>0$ and $C>0$ such that for any   $\phi_{b}\in (0,\delta_{0})$,
\begin{equation*}
\sum_{i=0}^3\|\partial _{x}^{i}UG^{i+2}+c_{i}\|_{L^\infty}\leq C
\phi_{b}
\end{equation*}
with
\begin{equation*}
U=-\tilde{\phi}, \ \ \tilde{n}-1, \ \ \log \tilde{n}, \ \ \frac{\tilde{u}}{u_{\infty}}-1, \ \  \frac{1}{\gamma}\left(\frac{\tilde{T}}{T_{\infty}}-1\right),
\end{equation*}
where $G=G(x)$ is a function of the form
\begin{equation}
\label{def.Gx}
G(x)=\Gamma x+\phi_{b}^{-\frac{1}{2}}.
\end{equation}
\end{itemize}
\end{proposition}
As in \cite{FFC, Riemann}, we introduce the {\it Bohm criterion} that corresponds to the condition that
\begin{eqnarray}\label{1.15}
&&\begin{aligned} &u^{2}_{\infty}\geq\frac{\gamma
RT_{\infty}+1}{m},\quad u_{\infty}<0.
\end{aligned}
\end{eqnarray}
From Proposition \ref{prop1.1}, we see that under the Bohm criterion, there exists a unique monotone small-amplitude stationary solution provided that either
\begin{equation}
\label{1.15a}
|\phi_b|\ll 1,\quad u_\infty<-\sqrt{\frac{\gamma
RT_{\infty}+1}{m}}
\end{equation}
or
\begin{equation}
\label{1.15b}
0<\phi_b\ll 1,\quad u_\infty=-\sqrt{\frac{\gamma
RT_{\infty}+1}{m}}.
\end{equation}
In both cases we call the monotone stationary solution the plasma sheath. From now on, we denote $(\tilde n, \tilde
u, \tilde T, \tilde\phi)$ to be the sheath solution to the half-space boundary-value problem \eqref{1.7} and \eqref{1.8} under the Bohm Criterion \eqref{1.15} additionally satisfying \eqref{1.15a} or \eqref{1.15b}.

The main concern of this paper is to study the asymptotic stability of the stationary solution $(\tilde n, \tilde
u, \tilde T, \tilde\phi)$. For this, it is convenient to
employ unknown functions $v:=\log n$ and $\tilde{v}:=\log \tilde{n}$
as well as perturbations
 $$(\varphi,\psi,\zeta,\sigma)(t,x)=(v,u,T,\phi)(t,x)-(\tilde v,\tilde u,\tilde T,\tilde \phi)(x).$$
From $\eqref{1.4}$ and $\eqref{1.7}$, we have
\begin{eqnarray}\label{1.16}
\begin{aligned}[b] \left(\begin{array} {ccc}
1 & 0 & 0\\
0& m & 0\\
0& 0 & 1\\
\end{array} \right)
\left(\begin{array} {c}
\varphi \\
\psi\\
\zeta\\
\end{array} \right)_{t}
&+\left(\begin{array} {ccc}
u & 1 & 0\\
RT& mu & R\\
0& (\gamma-1)T & u\\
\end{array} \right)
\left(\begin{array} {c}
\varphi \\
\psi\\
\zeta\\
\end{array} \right)_{x}\\[2mm]
&=-\left(\begin{array} {ccc}
\psi & 0 & 0\\
R\zeta& m\psi & 0\\
0& (\gamma-1)\zeta & \psi\\
\end{array} \right)
\left(\begin{array} {c}
\tilde{v} \\
\tilde{u}\\
\tilde{T}\\
\end{array} \right)_{x}+\left(\begin{array} {c}
0 \\
\sigma_{x}\\
0\\
\end{array} \right),
\end{aligned}
\end{eqnarray}
and
\begin{eqnarray}\label{1.17}
\begin{aligned} &
\sigma_{xx}=e^{\varphi+\tilde{v}}-e^{\tilde{v}}-e^{-(\sigma+\tilde{\phi})}+e^{-\tilde{\phi}}.
\end{aligned}
\end{eqnarray}
The initial and  boundary data to $\eqref{1.16}$-$\eqref{1.17}$ are
derived from  $\eqref{1.5}$--$\eqref{1.6}$ and $\eqref{1.8}$ as
\begin{eqnarray}\label{1.18}(\varphi,\psi,\zeta)(0,x)=(\varphi_{0},\psi_{0},\zeta_{0})(x):=(\log
n_{0}-\log \tilde n,u_{0}-\tilde u,T_{0}-\tilde T),
\end{eqnarray}
\begin{eqnarray}\label{1.19}
\lim_{x\rightarrow\infty}(\varphi_{0},\psi_{0},\zeta_{0})(x)=(0,0,0),\
\ \ \sigma(t,0)=0.
\end{eqnarray}
Provided that the perturbations are sufficiently small, both of the
characteristics of hyperbolic system  $\eqref{1.16}$ are negative
owing to $\eqref{1.15}$, namely,
\begin{equation*}
\left\{\begin{aligned}
&\lambda_{1}=\frac{(m+1)u-\sqrt{(m-1)^{2}u^{2}+4\gamma RT}}{2}<0,\\
&\lambda_{2}=u<0,\\
&\lambda_{3}=\frac{(m+1)u+\sqrt{(m-1)^{2}u^{2}+4\gamma RT}}{2}<0.
\end{aligned}\right.
\end{equation*}
Hence, no boundary conditions for the hyperbolic system $\eqref{1.16}$
are necessary for the well-posedness of the initial boundary value
problem \eqref{1.16}, \eqref{1.17}, \eqref{1.18} and \eqref{1.19}.

Before stating the main results, we first give the definition of the
function space $\mathscr{X}_{i}^{j}$ as follows:
$$
\mathscr{X}_{i}^{j}([0,M]):=\cap_{k=0}^{i}C^{k}([0,M];H^{j+i-k}(\mathbb{R}_{+})),
$$
$$
\mathscr{X}_{i}([0,M]):=\mathscr{X}_{i}^{0}([0,M]),
$$
for $i,j=0,1,2$, where $M>0$ is a constant.

\begin{theorem}[Nondegenerate case]\label{1.2theorem}
Assume that the condition  $\eqref{1.15a}$ holds.

\medskip
(i) Assume that the initial data satisfy
$$
(e^{\frac{\lambda
x}{2}}\varphi_{0},e^{\frac{\lambda x}{2}}\psi_{0},e^{\frac{\lambda
x}{2}}\zeta_{0})\in (H^2(\mathbb{R}_{+}))^{3}
$$
for some positive constant $\lambda$.
Then there exists a positive constant $\delta$ such that if
$\beta\in(0,\lambda]$ and
$$
\beta+(|\phi_{b}|+\|(e^{\frac{\beta
x}{2}}\varphi_{0},e^{\frac{\beta x}{2}}\psi_{0},e^{\frac{\beta
x}{2}}\zeta_{0})\|_{H^{2}})/\beta\leq\delta
$$
are satisfied, the initial
boundary value problem $\eqref{1.16}$--$\eqref{1.19}$ has a unique
solution as
$$
(e^{\frac{\beta x}{2}}\varphi,e^{\frac{\beta
x}{2}}\psi,e^{\frac{\beta x}{2}}\zeta,e^{\frac{\beta
x}{2}}\sigma)\in (\mathscr{X}_{2}(\mathbb{R}_{+}))^{3} \times
\mathscr{X}_{2}^{2}(\mathbb{R}_{+}).
$$
Moreover, the solution
$(\varphi,\psi,\zeta,\sigma)$ verifies the decay estimate
\begin{eqnarray*}
\|(e^{\frac{\beta x}{2}}\varphi,e^{\frac{\beta
x}{2}}\psi,e^{\frac{\beta
x}{2}}\zeta)(t)\|_{H^{2}}^{2}+\|e^{\frac{\beta
x}{2}}\sigma(t)\|_{H^{4}}^{2} \leq C\|(e^{\frac{\lambda
x}{2}}\varphi_{0},e^{\frac{\lambda x}{2}}\psi_{0},e^{\frac{\lambda
x}{2}}\zeta_{0})\|_{H^{2}}^{2}e^{-\mu t},
\end{eqnarray*}
where $C$ and $\mu$ are positive constants independent of $t$.

\medskip
(ii) Assume  $\lambda\geq 2$ holds. For an arbitrary
$\varepsilon\in(0,\lambda]$, there exists a positive constant
$\delta$ such that if
$$
((1+\beta
x)^{\frac{\lambda}{2}}\varphi_{0},(1+\beta
x)^{\frac{\lambda}{2}}\psi_{0},(1+\beta
x)^{\frac{\lambda}{2}}\zeta_{0})\in(H^2(\mathbb{R}_{+}))^{3}
$$
for $\beta>0$ and
$$
\beta+(|\phi_{b}|+\|((1+\beta
x)^{\frac{\lambda}{2}}\varphi_{0},(1+\beta
x)^{\frac{\lambda}{2}}\psi_{0},(1+\beta
x)^{\frac{\lambda}{2}}\zeta_{0})\|_{H^{2}})/\beta\leq\delta
$$
are satisfied, the initial boundary value problem
$\eqref{1.16}$--$\eqref{1.19}$ has a unique solution as
$$
((1+\beta
x)^{\frac{\varepsilon}{2}}\varphi,(1+\beta
x)^{\frac{\varepsilon}{2}}\psi,(1+\beta
x)^{\frac{\varepsilon}{2}}\zeta,(1+\beta
x)^{\frac{\varepsilon}{2}}\sigma)\in
(\mathscr{X}_{2}(\mathbb{R}_{+}))^{3} \times
\mathscr{X}_{2}^{2}(\mathbb{R}_{+}).
$$
Moreover, the solution
$(\varphi,\psi,\zeta,\sigma)$ verifies the decay estimate
\begin{equation*}
\begin{aligned}
&\|((1+\beta x)^{\frac{\varepsilon}{2}}\varphi,(1+\beta
x)^{\frac{\varepsilon}{2}}\psi,(1+\beta
x)^{\frac{\varepsilon}{2}}\zeta)(t)\|_{H^{2}}^{2}+\|(1+\beta
x)^{\frac{\varepsilon}{2}}\sigma(t)\|_{H^{4}}^{2}\\
& \leq C\|((1+\beta x)^{\frac{\lambda}{2}}\varphi_{0},(1+\beta
x)^{\frac{\lambda}{2}}\psi_{0},(1+\beta
x)^{\frac{\lambda}{2}}\zeta_{0})\|_{H^{2}}^{2}(1+\beta
t)^{-\lambda+\varepsilon},
\end{aligned}
\end{equation*}
where $C$ is a  positive
constant independent of $t$.
\end{theorem}

\begin{theorem}[Degenerate case]\label{1.3 theorem}
Assume that the condition  $\eqref{1.15b}$ holds. Let
$4<\lambda_{0}<5.5693\cdots$ be the unique real solution to the
equation
\begin{equation}
\label{def.la0}
\lambda_{0}(\lambda_{0}-1)(\lambda_{0}-2)-12\left(\frac{2}{\gamma+1}\lambda_{0}+2\right)=0,
\end{equation}
where $5.5693\cdots $ is the unique real solution to the equation
\begin{equation}
\label{def.5.5}
\lambda_{0}(\lambda_{0}-1)(\lambda_{0}-2)-12(\lambda_{0}+2)=0.
\end{equation}
Assume that
$\lambda\in[4,\lambda_{0})$ is satisfied. For arbitrary
$\varepsilon\in(0,\lambda]$ and $\theta\in(0,1],$ there exists a
positive constant $\delta$ such that if $\phi_{b}\in(0,\delta]$,
$\beta/(\Gamma \phi_{b}^{\frac{1}{2}} )\in [\theta,1]$,
$$
((1+\beta
x)^{\frac{\lambda}{2}}\varphi_{0},(1+\beta
x)^{\frac{\lambda}{2}}\psi_{0},(1+\beta
x)^{\frac{\lambda}{2}}\zeta_{0})\in(H^2(\mathbb{R}_{+}))^{3}
$$
and
$$
\|((1+\beta x)^{\frac{\lambda}{2}}\varphi_{0},(1+\beta
x)^{\frac{\lambda}{2}}\psi_{0},(1+\beta
x)^{\frac{\lambda}{2}}\zeta_{0})\|_{H^{2}}/\beta^{3}\leq\delta
$$
are satisfied, the initial boundary value problem
$\eqref{1.16}$--$\eqref{1.19}$ has a unique solution as
$$
((1+\beta
x)^{\frac{\varepsilon}{2}}\varphi,(1+\beta
x)^{\frac{\varepsilon}{2}}\psi,(1+\beta
x)^{\frac{\varepsilon}{2}}\zeta,(1+\beta
x)^{\frac{\varepsilon}{2}}\sigma)\in
(\mathscr{X}_{2}(\mathbb{R}_{+}))^{3} \times
\mathscr{X}_{2}^{2}(\mathbb{R}_{+}).
$$
Moreover, the solution
$(\varphi,\psi,\zeta,\sigma)$ verifies the decay estimate
\begin{equation*}
\begin{aligned}
&\|((1+\beta x)^{\frac{\varepsilon}{2}}\varphi,(1+\beta
x)^{\frac{\varepsilon}{2}}\psi,(1+\beta
x)^{\frac{\varepsilon}{2}}\zeta)(t)\|_{H^{2}}^{2}+\|(1+\beta
x)^{\frac{\varepsilon}{2}}\sigma(t)\|_{H^{4}}^{2}\\
& \leq C\|((1+\beta x)^{\frac{\lambda}{2}}\varphi_{0},(1+\beta
x)^{\frac{\lambda}{2}}\psi_{0},(1+\beta
x)^{\frac{\lambda}{2}}\zeta_{0})\|_{H^{2}}^{2}(1+\beta
t)^{-(\lambda-\varepsilon)/3},
\end{aligned}
\end{equation*}
where $C$ is a  positive
constant independent of $t$.
\end {theorem}

%


In what follows, we review some mathematical works related to the study of the subject in this paper and explain several crucial points in the proof of main results. First of all, for the stationary problem over a finite interval, Ambroso-M\'ehats-Raviart \cite{Asympto} studied the existence of monotone solutions by solving the Poisson equation with the small Debye length via the singular perturbation approach. Later, Ambroso \cite{Ambros} gave a further study to determine the stationary solutions in terms of different levels of an associated energy functional and numerically show which solution is asymptotically stable in large time. By simplifying the Euler-Poisson system on suitable physical regimes, Ha-Slemrod \cite{HS} presented a theory for the existence and dynamics of time-dependent sheaths with planar, cylindrical and spherical symmetry, where the sheath problem is formulated as a free boundary problem.

Through a delicate mathematical analysis, Suzuki \cite{Suzuki} first showed that the Bohm criterion gives a sufficient condition for an existence of the stationary solution by the phase plane method and also the stationary solution is time asymptotically stable under the weighted perturbations. Several extensions of \cite{Suzuki} have been made in \cite{NOS,boundaawa,multicomponent}. Specifically, Suzuki \cite{multicomponent}  further treated the same problem for a multicomponent plasma. Nishibata-Ohnawa-Suzuki \cite{NOS} refined the result in  \cite{Suzuki} by proving the stability exactly under the Bohm criterion in space dimensions up to three and also dealt with the degenerate case in which the Bohm criterion is marginally fulfilled.  Ohnawa \cite{boundaawa} studied the existence and asymptotic stability of boundary layers for the fluid-boundary interaction condition that the time change rate of the electric field $-\phi_x(t,0)$ at the boundary is equal to the total flux of charges.

Inspired by \cite{Suzuki} and \cite{NOS}, we expect to consider the effect of the variable temperature for the full Euler-Poisson system with the additional evolution equation of temperature function. Indeed, the existence of stationary solutions can be reduced to the isentropic case with the general $\gamma$-law pressure so that the proof is very similar to the isothermal Euler-Poisson system. For the asymptotic stability, we have to develop extra estimates to prove Theorem \ref{1.2theorem} and Theorem \ref{1.3 theorem}. Precisely, under the Bohm criterion,  we show that the stationary solution for the non-isentropic Euler-Poisson equations on a half line is time asymptotically stable provided that the initial perturbation is sufficiently small in the weighted Sobolev space. Moreover, the convergence rate of the solution toward the stationary
solution is obtained, provided that the initial perturbation belongs to the weighted Sobolev space. It seems that the non-isentropic case is more difficult to handle in analysis than the isentropic case. For example, in comparison with \cite{Suzuki}  and \cite{NOS}, we need to make additional efforts to consider the effect of the temperature equations in the proof. Technically, we observe that the only zero order dissipative term is associated with the weight parameter $\beta$, since this term arises from the integration by part of
the energy flux $I_{1}$ in \eqref{2.7}. In terms of the property of the stationary solution
in Proposition \ref{prop1.1}, two integral terms $I_{1}$ and $I_{2}$ in \eqref{2.7} should be added together to estimate. In fact, to estimate $I_{1}+I_{2}$ in the degenerate case \eqref{1.15b}, a key point is to derive the positive definiteness of the quadratic form $Q(x)$ in \eqref{2.1gh} that takes a complex form. The same situation occurs to the proof of Lemma \ref{main.result2.987} for the nondegenerate case \eqref{1.15a}.

In the end, we also mention \cite{Rousset, GHR, JKS, Jung} for the problem on the quasineutral limit of the Euler-Poisson system of the ions flow in the presence of boundaries and \cite{GGPS} for the derivation of the ions equations from the general two-fluid model in plasma physics. Meanwhile, for the Cauchy problem on the Euler-Poisson system of the similar form \eqref{1.1} for ions, we may refer to \cite{CG,GPu,LLS,LS,LY,Pu} and references therein for the extensive studies of the dispersive property.


The rest of the paper is arranged as follows. In Section 2, we give the energy estimates for the degenerate case. We make full use of the time-space weighted energy method to complete the proof of Theorem \ref{1.3 theorem}. In Section 3, we give the  energy estimates for the
nondegenerate case and complete the
proof of Theorem \ref{1.2theorem}. In the Appendix, we will give some basic results used in the proof of Proposition \ref{ste.pro1.1} and Proposition \ref{ste.pro2}.

\medskip
\noindent{\it Notations.} Throughout this paper, we denote a
positive constant (generally large) independent of $t$ by $C$. And
the character ``$C$" may take different values in different places.
 $L^p
= L^p(\mathbb{R}_{+}) \ (1 \leq p \leq \infty)$ denotes
 the usual Lebesgue space on $[0
,\infty)$ with its norm $ \|\cdot\|_ {L^p}$, and when $p=2$, we
write $ \| \cdot \| _{ L^2(\mathbb{R}_{+}) } = \| \cdot \|$. For a
nonnegative integer $s$,  $W^{s,p}$ denotes the usual $s$-th order
Sobolev space over $ [0,\infty)$ with its norm
$\|\cdot\|_{W^{s,p}}$. We use the abbreviation $H^s
(\mathbb{R}_{+})=W^{s,2}(\mathbb{R}_{+})$.
$C^{k}([0,T];H^s(\mathbb{R}_{+}))$ denotes the space of the $k$-times
continuously differential functions on the interval $[0,T]$ with
values in $H^s(\mathbb{R}_{+})$.
A norm with an algebraic weight is defined as follows:
$$\|f\|_{\alpha,\beta,i}:=\left(\int W_{\alpha,\beta}\sum_{j\leq
i}(\partial^{j}f)^{2}dx\right)^{\frac{1}{2}}, \ \ \ \ i,j\in
\mathbb{Z}, \ \ \ i,j\geq 0,$$
\begin{equation}
\label{def.W}
W_{\alpha,\beta}:=(1+\beta x)^{\alpha}, \ \ \ \ \alpha,\ \beta \in \mathbb{R},\ \ \ \beta>0.
\end{equation}
Note that this norm is equivalent to the norm defined by $\|(1+\beta
x)^{\frac{\alpha}{2}}f\|_{H^{i}}.$ The
 last subscript $i$ of $\|f\|_{\alpha,\beta,i}$ is often dropped for the case of $i=0$,
 namely, $\|f\|_{\alpha,\beta}:=\|f\|_{\alpha,\beta,0}$.

\section{Energy estimates in the degenerate case}
In this section, we study the asymptotic stability of the
stationary solution to $\eqref{1.1}$ for the degenerate case
$\eqref{1.15b}$, where the Bohm criterion is marginally fulfilled. In this case, we see from Proposition \ref{prop1.1} that the additional condition that $\phi_{b}>0$ is suitably small ensures the existence of a non-trivial monotone stationary solution to $\eqref{1.7}$ and $\eqref{1.8}$. To further show the dynamical stability of the stationary solution, we mainly focus on the {\it a priori} estimates that will be given in Proposition \ref{ste.pro1.1}. The global existence can be proved by the standard continuation argument based on the local
existence result together with the uniform {\it a priori} estimates. Here, the local-in-time existence can be proved by a similar method as in \cite{Suzuki} and we omit the  details for brevity.

In what follows we are devoted to establishing the {\it a priori} estimates in the degenerate case $\eqref{1.15b}$. For this purpose, we use the following notation for
convenience
\begin{eqnarray*}
 \mathcal
{N}_{\alpha,\beta}(M):=\sup_{0\leq t\leq
M}\|(\varphi,\psi,\zeta)(t)\|_{\alpha,\beta,2}.
\end{eqnarray*}

\begin{proposition}\label{ste.pro1.1}
Let the same conditions on
$T_{\infty}$, $u_{\infty}$, $\lambda_{0}$ and $\lambda$ as in
Theorem \ref{1.3 theorem} hold and let $(\varphi,\psi,\zeta,\sigma)$ be a solution to
$\eqref{1.16}$--$\eqref{1.19}$ over $[0,M]$ for $M>0$.
For any $\varepsilon\in(0,\lambda]$ and any $\theta\in(0,1]$,
there exist constants $\delta>0$ and $C>0$ independent of $M$ such that if all the following conditions
\begin{eqnarray}\label{a}
\phi_{b}\in(0,\delta],
\end{eqnarray}
\begin{eqnarray}\label{b}
\beta/(\Gamma \phi_{b}^{1/2}) \in [\theta,1],
\end{eqnarray}
\begin{eqnarray}\label{c}
((1+\beta x)^{\frac{\lambda}{2}}\varphi,(1+\beta
x)^{\frac{\lambda}{2}}\psi,(1+\beta
x)^{\frac{\lambda}{2}}\zeta,(1+\beta x)^{\frac{\lambda}{2}}\sigma)
\in (\mathscr{X}_{2}([0,M]))^{3}\times \mathscr{X}_{2}^{2}([0,M]),
\end{eqnarray}
and
\begin{eqnarray}\label{d}
\mathcal {N}_{\lambda,\beta}(M)/\beta^{3}\leq \delta
\end{eqnarray}
are satisfied, then it holds for any $0\leq t\leq M$
that
\begin{eqnarray}
\|(\varphi,\psi,\zeta)(t)\|_{\varepsilon,\beta,2}^{2}+\|\sigma(t)\|_{\varepsilon,\beta,4}^{2}\leq
C\|(\varphi_{0},\psi_{0},\zeta_{0})\|_{\lambda,\beta,2}^{2}(1+\beta
t)^{-(\lambda-\varepsilon)/3}.\label{prop2.1r1}
\end{eqnarray}
\end{proposition}

 For the proof of Proposition \ref{ste.pro1.1}, we need to first prove Lemma \ref{main.result2.9} which is a crucial step for deriving the {\it a priori} estimates on the zeroth order and first order space derivatives. After that, we give the estimates for the higher order derivatives in Lemma \ref{main.result2.erer9}. Proposition \ref{ste.pro1.1} is then proved by following Lemma \ref{main.result2.9} and  Lemma \ref{main.result2.erer9} at the end of this section.

\begin{lemma}\label{main.result2.9}
Under the same conditions as in Proposition \ref{ste.pro1.1}, there exist
positive constants $C$ and $\delta$ independent of $M$ such that if  conditions \eqref{a}, \eqref{b}, \eqref{c} and \eqref{d} are satisfied, it holds for
any $t\in[0,M]$ and any $\xi\geq 0$ that
 \begin{align}
&(1+\beta t)^{\xi}\|(\varphi,\psi,\zeta)(t)\|_{\varepsilon,\beta,1}^{2}\notag\\
&\quad+\int_{0}^{t} (1+\beta
\tau)^{\xi}\left[\beta^{3}\|(\varphi,\psi,\zeta)(\tau)\|_{\varepsilon-3,\beta}^{2}
+\beta\|(\varphi_{x},\psi_{x},\zeta_{x},\sigma_{x})(\tau)\|_{\varepsilon-1,\beta}^{2}\right]d\tau\notag\\
&\leq C\|(\varphi_{0},\psi_{0},\zeta_{0})\|_{\varepsilon,\beta,1}^{2}
+C\xi\beta\int_{0}^{t} (1+\beta
\tau)^{\xi-1}\|(\varphi,\psi,\zeta)(\tau)\|_{\varepsilon,\beta,1}^{2}d\tau.\label{2.10}
\end{align}
\end{lemma}

\begin{proof}
We start to derive from \eqref{1.16} and \eqref{1.17} several identities which will are used in the late energy estimates. First, it is convenient to rewrite $\eqref{1.16}$ as
\begin{eqnarray}\label{2.0}
\begin{aligned}[b]
\left(\begin{array} {ccc}
RT \ \ & 0 & \ \ 0\\
0 \ \ & m & \ \ 0\\
0 \ \ & 0 & \ \ \frac{R}{(\gamma-1)T} \\
\end{array} \right)
\left(\begin{array} {c}
\varphi \\
\psi\\
\zeta\\
\end{array} \right)_{t}&+\left(\begin{array} {ccc}
RTu \ \ & RT & \ \ 0\\
RT \ \ & mu & \ \ R\\
0 \ \ & R & \ \ \frac{R u}{(\gamma-1)T}\\
\end{array} \right)
\left(\begin{array} {c}
\varphi \\
\psi\\
\zeta\\
\end{array} \right)_{x}-\left(\begin{array} {c}
0 \\
\sigma\\
0\\
\end{array} \right)_{x}\\[2mm]
&=-\left(\begin{array} {ccc}
RT\psi \ \ & 0 & \ \ 0\\
R\zeta \ \ & m\psi & \ \ 0\\
0 \ \ & \frac{R\zeta}{T} & \ \ \frac{R\psi}{(\gamma-1)T}\\
\end{array} \right)
\left(\begin{array} {c}
\tilde{v} \\
\tilde{u}\\
\tilde{T}\\
\end{array} \right)_{x}.
\end{aligned}
\end{eqnarray}
Taking the inner product of $\eqref{2.0}$ with $\tilde{n}(\varphi,\psi,\zeta)$ and using $\tilde{v}_{x}=\frac{\tilde{n}_{x}}{\tilde{n}}$, one can get that
\begin{equation}
\label{2.1}
(\CE_0)_t+(\CH_0)_x+\CD_0+\tilde{n}\psi_{x}\sigma=\CR_0,
\end{equation}
where we have denoted
\begin{equation}
\label{def.e0}
\CE_0=\frac{\tilde{n}}{2}RT\varphi^{2}+\frac{\tilde{n}}{2}m\psi^{2}+\frac{\tilde{n}R}{2(\gamma-1)T}\zeta^{2},
\end{equation}
\begin{equation*}
\CH_0=\frac{\tilde{n}}{2}RTu\varphi^{2}+\tilde{n}RT\varphi\psi+\frac{\tilde{n}}{2}mu\psi^{2}+R\tilde{n}\zeta\psi+\frac{\tilde{n}R
u}{2(\gamma-1)T}\zeta^{2}
-\tilde{n}\sigma\psi,
\end{equation*}
\begin{eqnarray*}
\CD_0 & = & \left(-\frac{R T
u}{2}\tilde{n}_{x}-\frac{R \tilde{n}u}{2}\tilde{T}_{x}-\frac{R
\tilde{n}T}{2}\tilde{u}_{x}\right)\varphi^{2}-\tilde{n}R\tilde{T}_{x}\varphi\psi
+\left(\frac{m\tilde{n}}{2}\tilde{u}_{x}-\frac{mu}{2}\tilde{n}_{x}\right)\psi^{2}\\[2mm]
&&+\frac{R\tilde{n}}{(\gamma-1)T}\tilde{T}_{x}\zeta\psi+\tilde{n}_{x}\sigma\psi+\left(\frac{R\tilde{n}}{T}\tilde{u}_{x}
-\frac{R u
\tilde{n}_{x}+R\tilde{n}\tilde{u}_{x}}{2(\gamma-1)T}+\frac{R u
\tilde{n} \tilde{T}_{x}}{2(\gamma-1)T^{2}}\right)\zeta^{2},
\end{eqnarray*}
and
\begin{eqnarray}
\CR_0 &=&\left(\frac{\tilde{n}R}{2}\zeta_{t}+\frac{\tilde{n}R
u}{2}\zeta_{x}+\frac{\tilde{n}R
T}{2}\psi_{x}\right)\varphi^{2}+R\tilde{n}\zeta_{x}\varphi\psi+\frac{m\tilde{n}}{2}\psi_{x}\psi^{2}\notag\\[2mm]
&&+\left(\frac{R\tilde{n}\psi_{x}}{2(\gamma-1)T}-\frac{R\tilde{n}\zeta_{t}}{2(\gamma-1)T^{2}}-
\frac{R\tilde{n}u\zeta_{x}}{2(\gamma-1)T^{2}}\right)\zeta^{2}.
\label{def.r0}
\end{eqnarray}
Taking the one order $x$-derivative on $\eqref{2.0}$, further taking the inner product of the resulting system with $\tilde{n}(\varphi_{x},\psi_{x},\zeta_{x})$ and using $\tilde{v}_{x}=\frac{\tilde{n}_{x}}{\tilde{n}}$ again, similarly for obtaining \eqref{2.1}, one has
\begin{equation}
\label{2.2}
(\CE_1^{\rm x})_t+(\CH_1^{\rm x})_x-\tilde{n}\psi_{x}\sigma_{xx}=\CR_1^{\rm x},
\end{equation}
where we also have denoted
\begin{equation}
\label{def.e1x}
\CE_1^{\rm x}=\frac{\tilde{n}}{2}RT\varphi_{x}^{2}+\frac{\tilde{n}}{2}m\psi_{x}^{2}+\frac{\tilde{n}R}{2(\gamma-1)T}\zeta_{x}^{2},
\end{equation}
\begin{equation}
\label{def.h1x}
\CH_1^{\rm x}=\frac{\tilde{n}}{2}RTu\varphi_{x}^{2}+\tilde{n}RT\varphi_{x}\psi_{x}+\frac{\tilde{n}}{2}mu\psi_{x}^{2}+R\tilde{n}\zeta_{x}\psi_{x}+\frac{\tilde{n}R
u}{2(\gamma-1)T}\zeta_{x}^{2},
\end{equation}
and
\begin{eqnarray}
\CR_1^{\rm x}&=&\left[\frac{\tilde{n}R}{2}\zeta_{t}+ \frac{R T
u}{2}\tilde{n}_{x}-\frac{R \tilde{n}}{2}(T u)_{x}\right]\varphi_{x}^{2}
+\left[\frac{m
u}{2}\tilde{n}_{x}-\frac{m\tilde{n}}{2}u_{x}-m\tilde{n} \tilde{u}
_{x} \right]\psi_{x}^{2}\notag\\[2mm]
&&+\left[ \frac{R
u}{2(\gamma-1)T}\tilde{n}_{x}+\frac{R
\tilde{n}u}{2(\gamma-1)T^{2}}T_{x}-\frac{\tilde{n}R}{2(\gamma-1)T^{2}}\zeta_{t}-\frac{R
\tilde{n}}{2(\gamma-1)T}u_{x}-\frac{R
\tilde{n}}{T}\tilde{u}_{x}\right]\zeta_{x}^{2}\notag\\[2mm]
&&
 -R\tilde{n}T_{x}(\varphi_{t}+\psi
_{x})\varphi_{x}-R\tilde{n}_{x}\psi
 T_{x}\varphi_{x}+\frac{R\tilde{n}}{(\gamma-1)T^{2}}\zeta_{t}\zeta_{x}T_{x}
 +\frac{R\tilde{n}}{T^{2}}\tilde{u}_{x}\zeta\zeta_{x}T_{x}\notag\\[2mm]
&&+\frac{R\tilde{n}}{(\gamma-1)T^{2}}\tilde{T}_{x}\psi\zeta_{x}T_{x}-
\frac{R\tilde{n}}{(\gamma-1)T}\tilde{T}_{x}\psi_{x}\zeta_{x}
-R\tilde{n}\tilde{v}_{xx}(T\psi\varphi_{x}+\zeta\psi_{x})\notag\\[2mm]
&&-\tilde{n}\tilde{u}_{xx}(m\psi\psi_{x}+\frac{R}{T}\zeta\zeta_{x})
-\frac{R\tilde{n}}{(\gamma-1)T}\tilde{T}_{xx}\psi\zeta_{x}.
\label{def.r1x}
\end{eqnarray}

In what follows, we are going to multiply \eqref{2.1} by $e^{-\tilde{\phi}}$ and add the resulting equation together with \eqref{2.2}. To treating terms involving $\si$ in this process, we first notice that
 \begin{equation}
\label{2.3p1}
\tilde{n}\psi_{x}\sigma\cdot e^{-\tilde{\phi}}+(-\tilde{n}\psi_{x}\sigma_{xx})=\tilde{n}(e^{-\tilde{\phi}}\sigma-\sigma_{xx})\psi_{x}.
\end{equation}
Recall $\eqref{1.17}$. It follows from the Taylor expansion that
\begin{eqnarray}\label{2.3}
 \begin{aligned}
\sigma_{xx}=\tilde{n}(\varphi+\frac{1}{2}e^{\theta_{1}\varphi}\varphi^{2})+e^{-\tilde{\phi}}(\sigma-\frac{1}{2}e^{-\theta_{2}\sigma}\sigma^{2}),\
\ \  \theta_{1},\theta_{2}\in(0,1).
\end{aligned}
\end{eqnarray}
Substituting the above form into the right-hand side of \eqref{2.3p1} and using the first component equation of \eqref{1.16}, one has
\begin{eqnarray}
\tilde{n}\psi_{x}\sigma\cdot e^{-\tilde{\phi}}+(-\tilde{n}\psi_{x}\sigma_{xx})&=&(\frac{1}{2}\tilde{n}^{2}\varphi^{2})_{t}+
(\frac{1}{2}\tilde{n}^{2}u\varphi^{2})_{x}-\frac{1}{2}\tilde{n}^{2}\tilde{u}_{x}\varphi^{2}-\tilde{n}u\tilde{n}_{x}\varphi^{2}
+\tilde{n}^{2}\tilde{v}_{x}\varphi\psi\notag\\[2mm]
&&-\frac{1}{2}\tilde{n}^{2}\psi_{x}\varphi^{2}-\frac{\tilde{n}}{2}\left(\tilde{n}e^{\theta_{1}\varphi}\varphi^{2}
-e^{-(\theta_{2}\sigma+\tilde{\phi})}\sigma^{2}\right)\psi_{x}.
\label{2.4}
\end{eqnarray}
Therefore, taking the procedure \eqref{2.1}$\times e^{-\tilde{\phi}}+$\eqref{2.4} gives that
\begin{eqnarray}
&\dis\left(e^{-\tilde{\phi}}\CE_0
+\CE_1^{\rm x}+\frac{1}{2}\tilde{n}^{2}\varphi^{2}\right)_{t}
+\left(e^{-\tilde{\phi}}\CH_0+\CH_1^{\rm x}
+\frac{1}{2}\tilde{n}^{2}u\varphi^{2}\right)_{x}
+e^{-\tilde{\phi}}\tilde{\phi}_{x}\CH_0+e^{-\tilde{\phi}}\CD_0
\notag\\
&\dis+(-\tilde{n}u\tilde{n}_{x}-\frac{1}{2}\tilde{n}^{2}\tilde{u}_{x}
)\varphi^{2}+\tilde{n}^{2}\tilde{v}_{x}\varphi\psi =\CN_1,
\label{2.5}
\end{eqnarray}
where
\begin{equation}\label{2.6}
\CN_{1}=e^{-\tilde{\phi}}\CR_0+\CR_1^{\rm x}
+\frac{1}{2}\tilde{n}^{2}\psi_{x}\varphi^{2}+\frac{\tilde{n}}{2}\left(\tilde{n}e^{\theta_{1}\varphi}\varphi^{2}
-e^{-(\theta_{2}\sigma+\tilde{\phi})}\sigma^{2}\right)\psi_{x}.
\end{equation}

Recall \eqref{def.la0} and \eqref{def.5.5} for the definition of $\la_0$. Let $\la\in (4,\la_0]$ and then $\varepsilon\in (0,\la]$ be given.  We choose a space weight function $W_{\varepsilon,\beta}=(1+\beta x)^\varepsilon$ as in \eqref{def.W} for a suitable parameter $\beta>0$ depending on $\varepsilon$ to be determined later. Then, multiplying $\eqref{2.5}$ by $ W_{\varepsilon,\beta}$ and
integrating the resulting equation over $\mathbb{R}_{+}$,
one deduces that
\begin{eqnarray}\label{2.7}
\begin{aligned}[b]
&\frac{d}{dt}\int_{\mathbb{R}_{+}} W_{\varepsilon,\beta}
\left[e^{-\tilde{\phi}}\CE_0
+\CE_1^{\rm x}+\frac{1}{2}\tilde{n}^{2}\varphi^{2}\right]dx\\[2mm]
&\quad+\underbrace{\int_{\mathbb{R}_{+}}\varepsilon\beta W_{\varepsilon-1,\beta}\CL_1dx}_{I_{1}}
 +\int_{\mathbb{R}_{+}}\varepsilon\beta
W_{\varepsilon-1,\beta}(-\CH_1^{\rm x})
dx
+\underbrace{\int_{\mathbb{R}_{+}} W_{\varepsilon,\beta}\CL_{2} dx}_{I_{2}}\\[2mm]
&\quad-[e^{-\tilde{\phi}}\CH_0
+\frac{1}{2}\tilde{n}^{2}u\varphi^{2}](t,0)-\CH_1^{\rm x}(t,0)=\int_{\mathbb{R}_{+}} W_{\varepsilon,\beta}\CN_{1} dx,
\end{aligned}
\end{eqnarray}
where
\begin{equation*}
\CL_1=
-e^{-\tilde{\phi}}\CH_0-\frac{1}{2}\tilde{n}^{2}u\varphi^{2}
\end{equation*}
and
\begin{eqnarray*}
\CL_{2}&=&e^{-\tilde{\phi}}\tilde{\phi}_{x}\CH_0+e^{-\tilde{\phi}}\CD_0+(-\tilde{n}u\tilde{n}_{x}-\frac{1}{2}\tilde{n}^{2}\tilde{u}_{x}
)\varphi^{2}+\tilde{n}^{2}\tilde{v}_{x}\varphi\psi.
\end{eqnarray*}

Now we estimate each term in $\eqref{2.7}$ and we shall frequently use the results obtained in Section \ref{Append}. First, one can decompose $u$ and $T$ as $u=\psi+(\tilde{u}-u_{\infty})+u_{\infty}$ and  $T=\zeta+(\tilde{T}-T_{\infty})+T_{\infty}$, respectively. Recall Lemma \ref{1.1}(iv)
and Lemma \ref{main.result4.2}, as well as \eqref{def.h1x} for $\CH_1^{\rm x}$. Then one sees, under the condition $\eqref{1.15a}$ and
 $\eqref{a}$-$\eqref{d}$, that
\begin{eqnarray}\label{2.10a}
 \begin{aligned}[b]
&\frac{-\CH_1^{\rm x}}{\tilde{n}}=-\frac{1}{2}RTu\varphi_{x}^{2}-RT\varphi_{x}\psi_{x}-\frac{1}{2}mu\psi_{x}^{2}-R\zeta_{x}\psi_{x}-\frac{R
u}{2(\gamma-1)T}\zeta_{x}^{2}\\[2mm]
&\geq\frac{1}{2}RT_{\infty}(-u_{\infty})\varphi_{x}^{2}-RT_{\infty}\varphi_{x}\psi_{x}+\frac{1}{2}m(-u_{\infty})\psi_{x}^{2}
-R\zeta_{x}\psi_{x}+\frac{R
(-u_{\infty})}{2(\gamma-1)T_{\infty}}\zeta_{x}^{2}\\[2mm]
&\qquad-C(\mathcal
{N}_{\lambda,\beta}(M)+\phi_{b})(\varphi_{x}^{2}+\psi_{x}^{2}+\zeta_{x}^{2})\\[2mm]
&\geq (c-C\delta\beta)(\varphi_{x}^{2}+\psi_{x}^{2}+\zeta_{x}^{2}).
\end{aligned}
\end{eqnarray}
Therefore, it follows that the third term on the left-hand side of \eqref{2.7} can be estimated as
\begin{equation}\label{2.11a}
\int_{\mathbb{R}_{+}}\varepsilon\beta
W_{\varepsilon-1,\beta} (-\CH_1^{\rm x})dx
\geq
c\beta\|(\varphi_{x},\psi_{x},\zeta_{x})\|_{\varepsilon-1,\beta}^{2}.
\end{equation}
Here and in the sequel we have omitted the explicit dependence of $c>0$ on $\varepsilon$ for brevity and instead we would only emphasize the dependence of the constant coefficient on $\beta$.
In the same way as for treating \eqref{2.10a}, with the help of $u_{\infty}<0$ and the boundary condition $\sigma(t,0)=0$ as well as the smallness of $\delta>0$, for the boundary terms on the left-hand side of \eqref{2.7}, one has
\begin{equation}\label{2.11}
-\CH_1^{\rm x}(t,0)\geq 0
\end{equation}
and
\begin{equation}\label{2.11b}
-\left[e^{-\tilde{\phi}}\CH_0
+\frac{1}{2}\tilde{n}^{2}u\varphi^{2}\right](t,0)\geq 0.
\end{equation}

It remains to estimate two terms $I_{1}$ and
$I_{2}$ on the left-hand side of $\eqref{2.7}$. The
key is to make full use of properties of the stationary solution in Lemma \ref{1.1}(iv).  Through careful computations, one can capture the full energy dissipation of all the zero-order components with the positive coefficient. In fact, using Lemma \ref{1.1}(iv) and Lemma \ref{main.result4.2} together with the identity that $\tilde{n}(x)\tilde{u}(x)\equiv u_{\infty}$ and recalling $G=G(x)=\Ga x +\phi_b^{-1/2}$ as in \eqref{def.Gx} with the constant $\Gamma$ defined in \eqref{def.conGa}, one has
 \begin{align}
 I_{1} &\geq\int_{\mathbb{R}_{+}}\varepsilon\beta
W_{\varepsilon-1,\beta}\Big\{
\frac{1-G^{-2}}{2}(RT_{\infty}+1)|u_{\infty}|\varphi^{2}-(1-2G^{-2})RT_{\infty}\varphi\psi\notag\\[2mm]
&\qquad\qquad\qquad\qquad+\frac{1-G^{-2}}{2}m|u_{\infty}|\psi^{2}
+(1-G^{-2})\frac{R
|u_{\infty}|}{2(\gamma-1)T_{\infty}}\zeta^{2}\notag\\[2mm]
&\qquad\qquad\qquad\qquad-(1-2G^{-2})R\zeta\psi+(1-2G^{-2})\sigma\psi\Big\}dx\notag\\[2mm]
&\quad
-C\mathcal
{N}_{\lambda,\beta}(M)\|(\varphi,\psi,\zeta,\sigma)\|^{2}_{\varepsilon-3,\beta}\notag\\[2mm]
&\quad-C\phi_{b}\int_{\mathbb{R}_{+}}\beta
W_{\varepsilon-1,\beta}G^{-2}(\varphi^{2}+\psi^{2}+\zeta^{2}+\sigma^{2})dx,\label{2.12}
\end{align}
and
 \begin{align}
 I_{2}\geq&\int_{\mathbb{R}_{+}}
W_{\varepsilon,\beta}G^{-3}\Gamma|u_{\infty}|\left\{(\gamma
RT_{\infty}+1)\varphi^{2}+\frac{2(1-\gamma
RT_{\infty})}{|u_{\infty}|}\varphi\psi+3m\psi^{2}\right.\notag\\[2mm]
&\qquad\qquad\qquad\qquad\qquad\qquad\left.+\frac{4}{|u_{\infty}|}\sigma\psi+\frac{\gamma
R}{(\gamma-1)T_{\infty}}\zeta^{2}\right\}dx\notag\\[2mm] &
-C(\mathcal {N}_{\lambda,\beta}(M)+\phi_{b})\int_{\mathbb{R}_{+}}
W_{\varepsilon,\beta}G^{-3}(\varphi^{2}+\psi^{2}+\zeta^{2}+\sigma^{2})dx.\label{2.13}
\end{align}
Adding $\eqref{2.12}$ to  $\eqref{2.13}$ together and further using the Cauchy-Schwarz inequality
$$
\sigma\psi\geq-\left(\frac{|u_{\infty}|}{2}\sigma^{2}+\frac{1}{2|u_{\infty}|}\psi^{2}\right)
$$
and the condition $\eqref{1.15b}$, one has
\begin{align}
I_{1}+I_{2} &\geq I_{1,2}-C\mathcal
{N}_{\lambda,\beta}(M)\|(\varphi,\psi,\zeta,\sigma)\|^{2}_{\varepsilon-3,\beta}-C\phi_{b}\int_{\mathbb{R}_{+}}\beta
W_{\varepsilon-1,\beta}G^{-2}(\varphi^{2}+\psi^{2}+\zeta^{2}+\sigma^{2})dx\notag\\
&\quad-C(\mathcal
{N}_{\lambda,\beta}(M)+\phi_{b})\int_{\mathbb{R}_{+}}
W_{\varepsilon,\beta}G^{-3}(\varphi^{2}+\psi^{2}+\zeta^{2}+\sigma^{2})dx,\label{2.14}
\end{align}
where we have defined
\begin{eqnarray}\label{2.15}
 \begin{aligned}[b]
I_{1,2}=&\int_{\mathbb{R}_{+}}\left\{\frac{\varepsilon\beta}{2}W_{\varepsilon-1,\beta}(1-G^{-2})(RT_{\infty}+1)|u_{\infty}|+\Gamma|u_{\infty}|(\gamma
RT_{\infty}+1)W_{\varepsilon,\beta}G^{-3}\right\}\varphi^{2}dx\\[2mm] &
+\int_{\mathbb{R}_{+}}\left\{-RT_{\infty} \varepsilon\beta
W_{\varepsilon-1,\beta}(1-2G^{-2})+2\Gamma(1-\gamma
RT_{\infty})W_{\varepsilon,\beta}G^{-3} \right\}\varphi \psi dx\\[2mm] &
+
\frac{1}{|u_{\infty}|}\int_{\mathbb{R}_{+}}\left\{W_{\varepsilon-1,\beta}\frac{\varepsilon\beta}{2}[\gamma
RT_{\infty}+(1-\gamma RT_{\infty})G^{-2}]+\Gamma(3\gamma RT_{\infty}+1)W_{\varepsilon,\beta}G^{-3}\right\}\psi^{2}dx\\[2mm] &
-\frac{|u_{\infty}|}{2}\int_{\mathbb{R}_{+}}\left\{\varepsilon\beta
W_{\varepsilon-1,\beta}(1-2G^{-2})+4\Gamma W_{\varepsilon,\beta}G^{-3} \right\}\sigma^{2} dx\\[2mm] &
+\int_{\mathbb{R}_{+}}\left\{\frac{\varepsilon\beta}{2}W_{\varepsilon-1,\beta}\frac{R|u_{\infty}|}{(\gamma-1)T_{\infty}}(1-G(x)^{-2})
+\frac{\Gamma\gamma R|u_{\infty}|}{(\gamma-1)T_{\infty}}W_{\varepsilon,\beta}G^{-3}\right\}\zeta^{2}dx\\[2mm] &
-\int_{\mathbb{R}_{+}}\varepsilon\beta W_{\varepsilon-1,\beta} R(1-2G^{-2})\zeta
\psi dx.
\end{aligned}
\end{eqnarray}

Now we {\it claim} a key estimate on the coercivity of $I_{1,2}$ as follows:
\begin{eqnarray}\label{2.16}
 \begin{aligned}[b]
I_{1,2}\geq c \beta^{3}
\|(\varphi,\psi,\zeta)\|_{\varepsilon-3,\beta}^{2}+c\beta\|\sigma_{x}\|_{\varepsilon-1,\beta}^{2},
\end{aligned}
\end{eqnarray}
where as mentioned before, the constant $c>0$ may depend on $\varepsilon$ but not on $\beta$. Indeed, multiplying $\eqref{1.17}$ by $-\varepsilon\beta \sigma
W_{\varepsilon-1,\beta} $ and integrating the resulting equation over $\mathbb{R}_{+}$ with the help of Lemma \ref{1.1}(iv) and the Cauchy-Schwarz
inequality, it follows that
\begin{align*}
&\int_{\mathbb{R}_{+}} \varepsilon\beta
W_{\varepsilon-1,\beta}\left\{\sigma_{x}^{2}+\frac{1}{2}(1-2G^{-2})\sigma^{2}
\right\}dx\notag \\[2mm]
&\leq
\int_{\mathbb{R}_{+}}\frac{\varepsilon\beta}{2}W_{\varepsilon-1,\beta}(1-2G^{-2})\varphi^{2}dx+\int_{\mathbb{R}_{+}}\frac{1}{2}\varepsilon(\varepsilon-1)(\varepsilon-2)
\beta^{3}W_{\varepsilon-3,\beta}\varphi^{2}dx\notag\\[2mm]
&\quad+C(\beta^{2}+\phi_{b}+\mathcal{N}_{\lambda,\beta}(M)\beta^{-2})\beta^{3}\|\varphi\|_{\varepsilon-3,\beta}^{2}.
\end{align*}
Applying the above estimate into the fourth term on the right-hand side of $\eqref{2.15}$, we are able to obtain
 \begin{align}
I_{1,2}&\geq  \int_{\mathbb{R}_{+}}\beta
W_{\varepsilon-1,\beta}Q(x)dx+|u_{\infty}|\varepsilon\beta\|\sigma_{x}\|_{\varepsilon-1,\beta}^{2}\notag\\[2mm]
&\quad-C(\beta^{2}+\phi_{b}+\mathcal{N}_{\lambda,\beta}(M)\beta^{-2})\beta^{3}\|\varphi\|_{\varepsilon-3,\beta}^{2}, \label{2.18}
\end{align}
where $Q(x)$  is a quadratic form of $\varphi,$ $\psi$ and  $\zeta$ defined by
\begin{eqnarray}\label{2.1gh}
Q(x)=|u_{\infty}|q_{1}(x)\varphi^{2}+q_{2}(x)\varphi\psi+\frac{1}{|u_{\infty}|}q_{3}(x)\psi^{2}
+|u_{\infty}|q_{4}(x)\zeta^{2}+q_{5}(x)\zeta\psi,
\end{eqnarray}
with
\begin{eqnarray*}
q_{1}(x)&=&\frac{\varepsilon}{2}RT_{\infty}+B(x)^{-2}\Gamma^{-2}\Big\{\frac{(1-RT_{\infty})\varepsilon}{2}S(x)^{2}
+(\gamma
RT_{\infty}-1)S(x)^{3}\\[2mm]
&&\qquad\qquad\qquad\qquad\qquad-\frac{\Gamma^{2}}{2}\varepsilon(\varepsilon-1)(\varepsilon-2)
\Big\},\\[2mm]
q_{2}(x)&=&-RT_{\infty}\varepsilon+B(x)^{-2}\Gamma^{-2}\left\{2\varepsilon
RT_{\infty} S(x)^{2} +2(1-\gamma RT_{\infty})S(x)^{3} \right\},\\[2mm]
q_{3}(x)&=&\frac{\varepsilon}{2}\gamma
RT_{\infty}+B(x)^{-2}\Gamma^{-2}\left\{\frac{(1-\gamma
RT_{\infty})\varepsilon}{2}S(x)^{2} +(3\gamma RT_{\infty}+1)S(x)^{3}
\right\},\\[2mm]
q_{4}(x)&=&\frac{\varepsilon R}{2(\gamma-1)T_{\infty}}
+B(x)^{-2}\Gamma^{-2}\left\{-\frac{\varepsilon
R}{2(\gamma-1)T_{\infty}}S(x)^{2} +\frac{\gamma
R}{(\gamma-1)T_{\infty}}S(x)^{3} \right\},
\end{eqnarray*}
and
\begin{equation*}
q_{5}(x)=-\varepsilon R+2\varepsilon R
B(x)^{-2}\Gamma^{-2}S(x)^{2}.
\end{equation*}
Here, functions $B(x)$ and $S(x)$ are given by
\begin{eqnarray*}
B(x)=x+\beta^{-1},
\end{eqnarray*}
and
\begin{eqnarray*}
S(x)=(x+\beta^{-1})/(x+\Gamma^{-1}\phi_{b}^{-\frac{1}{2}}),
\end{eqnarray*}
respectively.
We claim that
\begin{eqnarray}\label{2.19}
 \begin{aligned}[b]
 q_{1}(x)>0,\ \ \  q_{3}(x)>0,\ \ \
 q_{4}(x)>0,
\end{aligned}
\end{eqnarray}
\begin{eqnarray}\label{2.19a}
 \begin{aligned}[b]
q_{2}(x)^{2}-4q_{1}(x)q_{3}(x)<0, \ \ \
q_{5}(x)^{2}-4q_{3}(x)q_{4}(x)<0,
\end{aligned}
\end{eqnarray}
and
\begin{eqnarray}\label{2.20}
 \begin{aligned}[b]
q_{1}(x)q_{5}(x)^{2}+q_{4}(x)q_{2}(x)^{2}-4q_{1}(x)q_{3}(x)q_{4}(x)
\leq -c B(x)^{-2}.
\end{aligned}
\end{eqnarray}
In fact, we observe from $\eqref{a}$-$\eqref{d}$ that
$$
S(x)\geq 1,\quad B(x)^{-2}\leq \beta^{2}\leq C
\phi_{b}\leq C \delta.
$$
Using the above observation and letting $\delta>0$ be small enough, it is straightforward to prove $\eqref{2.19}$ and $\eqref{2.19a}$; the details are omitted for brevity.
As for $\eqref{2.20}$, recalling \eqref{def.conGa},  one has
 \begin{align}
&q_{1}(x)q_{5}(x)^{2}+q_{4}(x)q_{2}(x)^{2}-4q_{1}(x)q_{3}(x)q_{4}(x)\notag \\[2mm]
&\leq
\frac{\varepsilon^{2}R^{2}}{2(\gamma-1)}B(x)^{-2}\Big\{
\varepsilon(\varepsilon-1)(\varepsilon-2)-2\varepsilon(1+\gamma R
T_{\infty})\Gamma^{-2}S(x)^{2}\notag \\[2mm]
&\qquad\qquad\qquad\qquad\quad-2[(\gamma^{2}+\gamma)R
T_{\infty}+2]\Gamma^{-2}S(x)^{3} \Big\}+C\beta^{2}B(x)^{-2}
\notag \\[2mm]
&=
C\beta^{2}B(x)^{-2}+\frac{\varepsilon^{2}R^{2}}{2(\gamma-1)}B(x)^{-2}\left\{
\varepsilon(\varepsilon-1)(\varepsilon-2)-12\varepsilon\frac{1+\gamma
R T_{\infty}}{1+\frac{\gamma^{2}+\gamma}{2}R T_{\infty}}S(x)^{2}-24S(x)^{3} \right\}\notag \\[2mm]
&\leq\frac{\varepsilon^{2}R^{2}}{2(\gamma-1)}B(x)^{-2}\left\{
\varepsilon(\varepsilon-1)(\varepsilon-2)-12\varepsilon\frac{2}{\gamma+1}S(x)^{2}-24S(x)^{3}+C\beta^{2} \right\}\notag \\[2mm]
&\leq \frac{\varepsilon^{2}R^{2}}{2(\gamma-1)}B(x)^{-2}\left\{
\varepsilon(\varepsilon-1)(\varepsilon-2)-12(\frac{2}{\gamma+1}\varepsilon+2)+C\beta^{2} \right\}.\label{2.20p1}
\end{align}
Recall \eqref{def.la0} for the definition of $\lambda_{0}$. Since $\la\in [4,\la_0)$ and $\varepsilon\in (0,\la]$, \eqref{2.20} follows from \eqref{2.20p1} by letting $\beta^2$ be small enough. Thus, combining \eqref{2.20} together with $\eqref{2.19}$ and $\eqref{2.20}$, it holds that
\begin{eqnarray*}
 \begin{aligned}[b]
Q(x)\geq c B(x)^{-2}(\varphi^{2}+\psi^{2}+\zeta^{2}),
\end{aligned}
\end{eqnarray*}
 which then implies that
\begin{equation}\label{2.22}
\int_{\mathbb{R}_{+}}\beta W_{\varepsilon-1,\beta}Q(x)dx\geq c\beta\int_{\mathbb{R}_{+}}
W_{\varepsilon-1,\beta} B(x)^{-2}(\varphi^{2}+\psi^{2}+\zeta^{2})dx =c
\beta^{3}\|(\varphi,\psi,\zeta)\|_{\varepsilon-3,\beta}^{2}.
\end{equation}
Therefore, the key estimate \eqref{2.16} follows by substituting $\eqref{2.22}$ into $\eqref{2.18}$ and letting  $\beta^{2}\leq
C\phi_{b}\leq C\delta$ and $\mathcal
{N}_{\lambda,\beta}(M)/\beta^{3}\leq \delta$ for $\delta>0$ small enough.

For the last three terms on the right-hand side of $\eqref{2.14}$, it is direct to obtain
\begin{align}
&C\mathcal
{N}_{\lambda,\beta}(M)\|(\varphi,\psi,\zeta,\sigma)\|^{2}_{\varepsilon-3,\beta}+C\phi_{b}\int_{\mathbb{R}_{+}}\beta
W_{\varepsilon-1,\beta}G(x)^{-2}(\varphi^{2}+\psi^{2}+\zeta^{2}+\sigma^{2})dx\notag\\
&\quad+C(\mathcal{N}_{\lambda,\beta}(M)+\phi_{b})\int_{\mathbb{R}_{+}}
W_{\varepsilon,\beta}G(x)^{-3}(\varphi^{2}+\psi^{2}+\zeta^{2}+\sigma^{2})dx\notag \\
&\leq C\delta \beta^{3}\|(\varphi,\psi,\zeta)\|_{\varepsilon-3,\beta}^{2}.\label{2.23}
\end{align}
 By substituting $\eqref{2.23}$ and $\eqref{2.16}$ into
 $\eqref{2.14}$, we have
\begin{eqnarray}\label{2.24}
 \begin{aligned}[b]
I_{1}+I_{2}\geq
(c-C\delta)(\beta^{3}\|(\varphi,\psi,\zeta)\|_{\varepsilon-3,\beta}^{2}+\beta\|\sigma_{x}\|_{\varepsilon-1,\beta}^{2}).
\end{aligned}
\end{eqnarray}

At the end, we estimate the only term on the right-hand side of
$\eqref{2.7}$. In fact, recalling \eqref{2.6} as well as \eqref{def.r0} and \eqref{def.r1x}, it holds that
\begin{eqnarray}\label{2.25}
 \begin{aligned}[b]
\int_{\mathbb{R}_{+}} W_{\varepsilon,\beta}\CN_{1}dx\leq  C\delta
\left\{\beta^{3}\|(\varphi,\psi,\zeta)\|_{\varepsilon-3,\beta}^{2}+\beta\|(\varphi_{x},\psi_{x},\zeta_{x})\|_{\varepsilon-1,\beta}^{2}\right\},
\end{aligned}
\end{eqnarray}
where we have used Lemma \ref{1.1}(iv), $\eqref{1.16}$, $\eqref{a}$-$\eqref{d}$, $\lambda\geq 4$ and the
Cauchy-Schwarz inequality and the elliptic estimate in Lemma \ref{main.result4}.

 Substituting $\eqref{2.11}$, $\eqref{2.11b}$, $\eqref{2.24}$ and $\eqref{2.25}$
into $\eqref{2.7}$, we have
\begin{align}
&\frac{d}{dt}\int_{\mathbb{R}_{+}} W_{\varepsilon,\beta}(e^{-\tilde{\phi}}\CE_0
+\CE_1^{\rm x}+\frac{1}{2}\tilde{n}^{2}\varphi^{2})dx\notag\\[2mm]
&\quad +c\beta^{3}\|(\varphi,\psi,\zeta)\|_{\varepsilon-3,\beta}^{2}+c\beta\|(\varphi_{x},\psi_{x},\zeta_{x},\sigma_{x})\|_{\varepsilon-1,\beta}^{2}
\leq 0,\label{2.26}
\end{align}
provided that $\delta>0$ is sufficiently small, where $\CE_0$ and $\CE_1^{\rm x}$ are defined in \eqref{def.e0} and \eqref{def.e1x} respectively. Furthermore, multiplying $\eqref{2.26}$ by $(1+\beta \tau)^{\xi}$ and integrating the resulting inequality
over $(0,t)$ give the desired estimate $\eqref{2.10}$.  This hence completes the proof of Lemma \ref{main.result2.9}.
\end{proof}

\begin{lemma}\label{main.result2.erer9}
Under the same conditions as in Proposition \ref{ste.pro1.1}, there exist
positive constants $C$ and $\delta$ independent of $M$ such that if
conditions $\eqref{a}$--$\eqref{d}$ are satisfied, it holds for
any $t\in[0,M]$ and any $\xi\geq 0$ that
 \begin{align}
 &(1+\beta t)^{\xi}\|(\varphi_{t},\psi_{t},\zeta_{t})(t)\|_{\varepsilon,\beta,1}^{2}\notag\\
&\quad +\int_{0}^{t} (1+\beta
 \tau)^{\xi}\left\{\beta^{3}\|(\varphi_{t},\psi_{t},\zeta_{t})\|_{\varepsilon-3,\beta}^{2}
 +\beta\|(\varphi_{tx},\psi_{tx},\zeta_{tx},\sigma_{tx})\|_{\varepsilon-1,\beta}^{2}\right\}d\tau\notag\\
& \leq C\|(\varphi_{0t},\psi_{0t},\zeta_{0t})\|_{\varepsilon,\beta,1}^{2}
 +C\xi\beta\int_{0}^{t} (1+\beta
 \tau)^{\xi-1}\|(\varphi_{t},\psi_{t},\zeta_{t})\|_{\varepsilon,\beta,1}^{2}
 d\tau\notag\\
 &\quad+C\delta\int_{0}^{t}(1+\beta
\tau)^{\xi}\left[
\beta\|(\varphi_{x},\psi_{x},\zeta_{x})\|_{\varepsilon-1,\beta}^{2}+\beta^{3}\|(\varphi,\psi,\zeta,\varphi_{xx},\psi_{xx},\zeta_{xx})\|_{\varepsilon-3,\beta}^{2}\right]d\tau.
\label{2.35}
 \end{align}
\end{lemma}

\begin{proof}
We follow the same steps as in deriving $\eqref{2.10}$ in the proof of Lemma \ref{main.result2.9}. On one hand, taking the time derivative on $\eqref{2.0}$, then taking the inner product of the resulting system with $\tilde{n}(\varphi_{t},\psi_{t},\zeta_{t})$ and using $\tilde{v}_{x}=\frac{\tilde{n}_{x}}{\tilde{n}}$, it follows that
\begin{equation}\label{2.27}
(\CE_1^{\rm t})_t+(\CH_1^{\rm t})_{x}+\CD_1^{\rm t}+\tilde{n}\psi_{xt}\sigma_{t}=\CR_1^{\rm t},
\end{equation}
where we have denoted
\begin{equation*}
\CE_1^{\rm t}=\frac{\tilde{n}}{2}RT\varphi_{t}^{2}+\frac{\tilde{n}}{2}m\psi_{t}^{2}+\frac{\tilde{n}R}{2(\gamma-1)T}\zeta_{t}^{2},
\end{equation*}
\begin{equation*}
\CH_1^{\rm t}=\frac{\tilde{n}}{2}RTu\varphi_{t}^{2}+\tilde{n}RT\varphi_{t}\psi_{t}+\frac{\tilde{n}}{2}mu\psi_{t}^{2}+R\tilde{n}\zeta_{t}\psi_{t}+\frac{\tilde{n}R
u}{2(\gamma-1)T}\zeta_{t}^{2}
-\tilde{n}\sigma_{t}\psi_{t},
\end{equation*}
\begin{eqnarray*}
\CD_1^{\rm t}&=&\left(-\frac{R T
u}{2}\tilde{n}_{x}-\frac{R \tilde{n}u}{2}\tilde{T}_{x}-\frac{R
\tilde{n}T}{2}\tilde{u}_{x}\right)\varphi_{t}^{2}-\tilde{n}R\tilde{T}_{x}\varphi_{t}\psi_{t}\\[2mm]
&&+\left(\frac{m\tilde{n}}{2}\tilde{u}_{x}-\frac{mu}{2}\tilde{n}_{x}\right)\psi_{t}^{2}
+\frac{R\tilde{n}}{(\gamma-1)T}\tilde{T}_{x}\zeta_{t}\psi_{t}+\tilde{n}_{x}\sigma_{t}\psi_{t}\\[2mm]
&&+\left[\frac{R\tilde{n}}{T}\tilde{u}_{x}
-\frac{Ru\tilde{n}_{x}+R\tilde{n}\tilde{u}_{x}}{2(\gamma-1)T}+\frac{R
u\tilde{n} \tilde{T}_{x}}{2(\gamma-1)T^{2}}\right]\zeta_{t}^{2},
\end{eqnarray*}
and
\begin{align*}
\CR_1^{\rm t}&=\left(-\frac{\tilde{n}R}{2}\zeta_{t}+\frac{\tilde{n}R
u}{2}\zeta_{x}+\frac{\tilde{n}R T}{2}\psi_{x}\right)\varphi_{t}^{2}
+R\tilde{n}\zeta_{x}\varphi_{t}\psi_{t}-\frac{m\tilde{n}}{2}\psi_{x}\psi_{t}^{2}\\[3mm]
&\quad-(\tilde{n}Ru\zeta_{t}+\tilde{n}R
T\psi_{t})\varphi_{x}\varphi_{t}-\frac{R \tilde{n} \psi_{t}\zeta_{x}\zeta_{t}}{(\gamma-1)T}\\[3mm]
&\quad
+R\tilde{n}\left(\frac{\psi_{x}}{2(\gamma-1)T}+\frac{\zeta_{t}+u\zeta_{x}}{2(\gamma-1)T^{2}}
+\frac{\tilde{u}_{x}\zeta}{T^{2}}
+\frac{\tilde{T}_{x}\psi}{(\gamma-1)T^{2}}\right)\zeta_{t}^{2}\\[3mm]
&\quad-\tilde{n}R\zeta_{t}(\psi_{x}\varphi_{t}+\psi_{t}\varphi_{x})-R\tilde{n}_{x}\psi\zeta_{t}\varphi_{t}.
\end{align*}
On the other hand, taking derivatives with respect to $x$ and $t$ about $\eqref{2.0}$, then taking the inner product of the resulting system with $\tilde{n}(\varphi_{xt},\psi_{xt},\zeta_{xt})$ and using $\tilde{v}_{x}=\frac{\tilde{n}_{x}}{\tilde{n}}$ again, it also follows that
\begin{equation}\label{2.28}
(\CE_2^{\rm xt})_t+(\CH_2^{\rm xt})_x
-\tilde{n}\psi_{xt}\sigma_{xxt}=\CR_2^{\rm xt},
\end{equation}
where we have denoted
\begin{equation*}
\CE_2^{\rm xt}=\frac{\tilde{n}}{2}RT\varphi_{xt}^{2}+\frac{\tilde{n}}{2}m\psi_{xt}^{2}+\frac{\tilde{n}R}{2(\gamma-1)T}\zeta_{xt}^{2},
\end{equation*}
\begin{equation*}
\CH_2^{\rm xt}=\frac{\tilde{n}}{2}RTu\varphi_{xt}^{2}+\tilde{n}RT\varphi_{xt}\psi_{xt}+\frac{\tilde{n}}{2}mu\psi_{xt}^{2}+R\tilde{n}\zeta_{xt}\psi_{xt}+\frac{\tilde{n}R
u}{2(\gamma-1)T}\zeta_{xt}^{2},
\end{equation*}
and
\begin{equation*}
\CR_2^{\rm xt}=\CR_{2,1}^{\rm xt}+\CR_{2,2}^{\rm xt}
\end{equation*}
with
\begin{align}
\CR_{2,1}^{\rm xt}
&=
\left[-\frac{\tilde{n}R}{2}\zeta_{t}+ \frac{R T
u}{2}\tilde{n}_{x}-\frac{R \tilde{n}}{2}(T
u)_{x}\right]\varphi_{xt}^{2} +\left[\frac{m
u}{2}\tilde{n}_{x}-\frac{m\tilde{n}}{2}u_{x}-m\tilde{n} \tilde{u}
_{x} \right]\psi_{xt}^{2}\notag \\[3mm]
&\quad+\left[ \frac{R
u}{2(\gamma-1)T}\tilde{n}_{x}+\frac{R
\tilde{n}u}{2(\gamma-1)T^{2}}T_{x}+\frac{\tilde{n}R}{2(\gamma-1)T^{2}}\zeta_{t}-\frac{R
\tilde{n}}{2(\gamma-1)T}u_{x}\right]\zeta_{xt}^{2}\notag\\[3mm]
&\quad
 -R\tilde{n}T_{x}(\varphi_{tt}+\psi
_{xt})\varphi_{xt}+ RT\tilde{n}_{x}\psi
_{xt}\varphi_{xt}+\frac{R\tilde{n}T_{x}}{(\gamma-1)T^{2}}\zeta_{tt}\zeta_{xt}\notag\\[3mm]
&\quad-\tilde{n}\tilde{u}_{xx}\left[m\psi_{t}\psi_{xt}+\left(\frac{R\zeta}{T}\right)_{t}\zeta_{xt}\right]
-R\tilde{n}\varphi_{xt}\varphi_{t}\zeta_{xt}
-\frac{R}{\gamma-1}\left(\frac{\tilde{n}}{T}\right)_{xt}\zeta_{xt}\zeta_{t}\notag
\end{align}
and
\begin{align*}
\CR_{2,2}^{\rm xt}&=
-\tilde{n}(RTu)_{xt}\varphi_{xt}\varphi_{x}
-\tilde{n}m\psi_{x}\psi_{xt}^{2}-\frac{R\tilde{n}}{\gamma-1}\left(\frac{u}{T}\right)_{xt}\zeta_{xt}\zeta_{x}
-R\tilde{n}(\varphi_{xt}\psi_{x}+\psi_{xt}\varphi_{x})\zeta_{xt}\notag\\[2mm]
&\quad-\tilde{n}(RTu)_{t}\varphi_{xt}\varphi_{xx}
-R\tilde{n}(\varphi_{xt}\psi_{xx}+\psi_{xt}\varphi_{xx})\zeta_{t}-\tilde{n}m\psi_{t}\psi_{xt}\psi_{xx}\notag\\[2mm]
&\quad
-\frac{R\tilde{n}}{\gamma-1}\left(\frac{u}{T}\right)_{t}\zeta_{xt}\zeta_{xx}-\tilde{n}_{x}(RT\psi)_{xt}\varphi_{xt}
-\tilde{n}\tilde{u}_{x}\left(\frac{R\zeta}{T}\right)_{xt}\zeta_{xt}-\frac{R\tilde{n}}{\gamma-1}\tilde{T}_{x}
\left(\frac{\psi}{T}\right)_{xt}\zeta_{xt}\notag\\[2mm]
&\quad
-\tilde{n}\tilde{v}_{xx}\left[(RT\psi)_{t}\varphi_{xt}+R\zeta_{t}\psi_{xt}\right]-\frac{R\tilde{n}}{\gamma-1}\tilde{T}_{xx}\left(\frac{\psi}{T}\right)_{t}\zeta_{xt}.
\end{align*}

In terms of the identities \eqref{2.27} and \eqref{2.28} as above, we are going to take their combination by the sum of \eqref{2.27} multiplied by $e^{-\tilde{\phi}}$ together with \eqref{2.28}. Notice that
\begin{equation}
\label{2.29p1}
\tilde{n}\psi_{xt}\sigma_{t}\cdot e^{-\tilde{\phi}}-\tilde{n}\psi_{xt}\sigma_{xxt}=\tilde{n}\left(e^{-\tilde{\phi}}\sigma_{t}-\sigma_{xxt}\right)\psi_{xt}.
\end{equation}
Taking the time derivative on $\eqref{2.3}$, one has
\begin{eqnarray}\label{2.29}
 \begin{aligned}
e^{-\tilde{\phi}}\sigma_{t}-\sigma_{xxt}=-\tilde{n}\varphi_{t}-\frac{1}{2}\left(\tilde{n}e^{\theta_{1}\varphi}\varphi^{2}-e^{-(\theta_{2}\sigma+\tilde{\phi})}\sigma^{2}\right)_{t},\
\ \  \theta_{1},\theta_{2}\in(0,1).
\end{aligned}
\end{eqnarray}
Substituting $\eqref{2.29}$ into \eqref{2.29p1} and
using the first component equation of \eqref{1.16}, one gets
\begin{align}
&\tilde{n}\psi_{xt}\sigma_{t}\cdot e^{-\tilde{\phi}}-\tilde{n}\psi_{xt}\sigma_{xxt}\notag \\[2mm]
&=\left(\frac{1}{2}\tilde{n}^{2}\varphi_{t}^{2}\right)_{t}
+\left(\frac{1}{2}\tilde{n}^{2}u\varphi_{t}^{2}\right)_{x}
-\frac{1}{2}\tilde{n}^{2}\tilde{u}_{x}\varphi_{t}^{2}
-\tilde{n}u\tilde{n}_{x}\varphi_{t}^{2}
+\tilde{n}^{2}\tilde{v}_{x}\varphi_{t}\psi_{t}
-\frac{1}{2}\tilde{n}^{2}\psi_{x}\varphi_{t}^{2}\notag\\[2mm]
&\quad-\frac{\tilde{n}}{2}\left(\tilde{n}e^{\theta_{1}\varphi}\varphi^{2}
-e^{-(\theta_{2}\sigma+\tilde{\phi})}\sigma^{2}\right)_{t}\psi_{xt}
+\tilde{n}^{2}\psi_{t}\varphi_{x}\varphi_{t}.
\label{2.30}
\end{align}
Then, with the help of $\eqref{2.30}$, multiplying $\eqref{2.27}$ by $e^{-\tilde{\phi}}$ and adding the
resulting equation to $\eqref{2.28}$ give that
\begin{eqnarray}
&\dis \left(e^{-\tilde{\phi}}\CE_1^{\rm t}+\CE_2^{\rm xt}+\frac{1}{2}\tilde{n}^{2}\varphi_{t}^{2}\right)_{t}
+\left(e^{-\tilde{\phi}}\CH_1^{\rm t}+\CH_2^{\rm xt}
+\frac{1}{2}\tilde{n}^{2}u\varphi_{t}^{2}\right)_{x}
+e^{-\tilde{\phi}}\tilde{\phi}_{x}\CH_1^{\rm t}+e^{-\tilde{\phi}}\CD_1^{\rm t}\notag \\[2mm]
&\dis +(-\tilde{n}u\tilde{n}_{x}-\frac{1}{2}\tilde{n}^{2}\tilde{u}_{x}
)\varphi_{t}^{2}
+\tilde{n}^{2}\tilde{v}_{x}\varphi_{t}\psi_{t}=\CN_{2},
\label{2.31}
\end{eqnarray}
where
\begin{equation*}
\CN_2 =e^{-\tilde{\phi}}\CR_1^{\rm t}+\CR_2^{\rm xt}
+\frac{1}{2}\tilde{n}^{2}\psi_{x}\varphi_{t}^{2}
+\frac{\tilde{n}}{2}\left(\tilde{n}e^{\theta_{1}\varphi}\varphi^{2}
-e^{-(\theta_{2}\sigma+\tilde{\phi})}\sigma^{2}\right)_{t}\psi_{xt}
-\tilde{n}^{2}\psi_{t}\varphi_{x}\varphi_{t}.
\end{equation*}
Similarly for deriving \eqref{2.7}, it follows from \eqref{2.31} that
\begin{align}
&\frac{d}{dt}\int_{\mathbb{R}_{+}} W_{\varepsilon,\beta}
\left[e^{-\tilde{\phi}}\CE_1^{\rm t}+\CE_2^{\rm xt}+\frac{1}{2}\tilde{n}^{2}\varphi_{t}^{2}\right]dx\notag\\[3mm]
&\quad+\underbrace{\int_{\mathbb{R}_{+}}\varepsilon\beta W_{\varepsilon-1,\beta}\CL_3dx}_{I_{3}}
+\int_{\mathbb{R}_{+}}\varepsilon\beta
W_{\varepsilon-1,\beta} (-\CH_2^{\rm xt})
dx+\underbrace{\int_{\mathbb{R}_{+}} W_{\varepsilon,\beta}\CL_{4} dx}_{I_{4}}\notag\\[3mm]
&\quad
-[e^{-\tilde{\phi}}\CH_1^{\rm t}+\CH_2^{\rm xt}
+\frac{1}{2}\tilde{n}^{2}u\varphi_{t}^{2}](t,0)
=\int_{\mathbb{R}_{+}} W_{\varepsilon,\beta}\CN_{2} dx,
\label{2.33}
\end{align}
where we have denoted
\begin{equation*}
\CL_3=
-e^{-\tilde{\phi}}\CH_1^{\rm t}-\frac{1}{2}\tilde{n}^{2}u\varphi_{t}^{2}
\end{equation*}
and
 \begin{align*}
\CL_{4}=&e^{-\tilde{\phi}}\tilde{\phi}_x\CH_1^t +e^{-\tilde{\phi}}\CD_1^t+\left(-\tilde{n}u\tilde{n}_{x}-\frac{1}{2}\tilde{n}^{2}\tilde{u}_{x}
\right)\varphi_{t}^{2}
+\tilde{n}^{2}\tilde{v}_{x}\varphi_{t}\psi_{t}.
\end{align*}

 We now make estimates on each term in \eqref{2.33} in the same way as for treating \eqref{2.7}. First, with the help of $u_{\infty}<0$ as well as smallness of $\delta>0$, we notice that
 \begin{align}
\frac{-\CH_2^{\rm xt}}{\tilde{n}}= & -\frac{1}{2}RTu\varphi_{xt}^{2}-RT\varphi_{xt}\psi_{xt}-\frac{1}{2}mu\psi_{xt}^{2}-R\zeta_{xt}\psi_{xt}-\frac{R
u}{2(\gamma-1)T}\zeta_{xt}^{2}\notag\\[2mm]
\geq & \frac{1}{2}RT_{\infty}(-u_{\infty})\varphi_{xt}^{2}-RT_{\infty}\varphi_{xt}\psi_{xt}+\frac{1}{2}m(-u_{\infty})\psi_{xt}^{2}
-R\zeta_{xt}\psi_{xt}+\frac{R
(-u_{\infty})}{2(\gamma-1)T_{\infty}}\zeta_{xt}^{2}\notag\\[2mm]
&-C(\mathcal
{N}_{\lambda,\beta}(M)+\phi_{b})(\varphi_{xt}^{2}+\psi_{xt}^{2}+\zeta_{xt}^{2})\notag\\[2mm]
\geq & (c-C\delta\beta)(\varphi_{xt}^{2}+\psi_{xt}^{2}+\zeta_{xt}^{2}), \notag
\end{align}
so it holds that
\begin{equation}
\label{2.37a}
\int_{\mathbb{R}_{+}}\varepsilon\beta
W_{\varepsilon-1,\beta} (-\CH_2^{\rm xt})
dx\geq
c\beta\|(\varphi_{xt},\psi_{xt},\zeta_{xt})\|_{\varepsilon-1,\beta}^{2}.
\end{equation}
Similarly, for the boundary terms, one has
\begin{equation}
\label{2.37}
-\left[e^{-\tilde{\phi}}\CH_1^{\rm t}
+\frac{1}{2}\tilde{n}^{2}u\varphi_{t}^{2}\right](t,0)\geq 0,\quad -\CH_2^{\rm xt}(t,0)\geq 0.
\end{equation}
The estimates of $I_{3}$ and $I_{4}$ are similar to those of
$I_{1}$ and $I_{2}$. In fact, we only need to use $(\varphi_{t},\psi_{t},\zeta_{t},\sigma_{t})$
to replace $(\varphi,\psi,\zeta,\sigma)$ in all the estimates on $I_{1}$ and $I_{2}$. Thus, as for obtaining \eqref{2.24}, it holds that
\begin{equation}
I_{3}+I_{4}\geq
(c-C\delta)\left\{\beta^{3}\|(\varphi_{t},\psi_{t},\zeta_{t})\|_{\varepsilon-3,\beta}^{2}+\beta\|\sigma_{xt}\|_{\varepsilon-1,\beta}^{2}\right\},\label{2.50}
\end{equation}
where $c>0$ is a constant depending on $\varepsilon$ but independent of $\beta$. For the term on the right-hand side of \eqref{2.33}, similar to \eqref{2.25}, direct computations give that
\begin{eqnarray}
\int_{\mathbb{R}_{+}} W_{\varepsilon,\beta}\CN_{2} dx & \leq & C\delta
\beta\|(\varphi_{x},\psi_{x},\zeta_{x},\varphi_{tx},\psi_{tx},\zeta_{tx})\|_{\varepsilon-1,\beta}^{2}
\notag\\
&&+C\delta\beta^{3}\|(\varphi,\psi,\zeta,\varphi_{t},\psi_{t},\zeta_{t},\varphi_{xx},\psi_{xx},\zeta_{xx})\|_{\varepsilon-3,\beta}^{2}.
\label{2.51}
\end{eqnarray}

Now, substituting all the estimates \eqref{2.37a}, \eqref{2.37}, \eqref{2.50} and \eqref{2.51}
into $\eqref{2.33}$, one can derive that
\begin{align}
&\frac{d}{dt}\int_{\mathbb{R}_{+}} W_{\varepsilon,\beta}
\left[e^{-\tilde{\phi}}\CE_1^{\rm t}+\CE_2^{\rm xt}+\frac{1}{2}\tilde{n}^{2}\varphi_{t}^{2}\right]dx\notag\\[2mm]
&\quad
+c\beta^{3}\|(\varphi_{t},\psi_{t},\zeta_{t})\|_{\varepsilon-3,\beta}^{2}
+c\beta\|(\varphi_{tx},\psi_{tx},\zeta_{tx},\sigma_{tx})\|_{\varepsilon-1,\beta}^{2}\notag\\[2mm]
&\leq  C\delta
\beta\|(\varphi_{x},\psi_{x},\zeta_{x})\|_{\varepsilon-1,\beta}^{2}+C\delta\beta^{3}\|(\varphi,\psi,\zeta,\varphi_{xx},\psi_{xx},\zeta_{xx})\|_{\varepsilon-3,\beta}^{2},
\label{2.52}
\end{align}
provided that $\delta>0$ is sufficiently small. Furthermore, the desired estimate \eqref{2.35} follows by multiplying \eqref{2.52} by $(1+\beta \tau)^{\xi}$ and integrating the resulting inequality
over $(0,t)$. Therefore, we complete the proof of Lemma \ref{main.result2.erer9}.
\end{proof}

\begin{proof}[Proof of Proposition \ref{ste.pro1.1}]
Now, following Lemma \ref{main.result2.9} and Lemma \ref{main.result2.erer9} above, we are ready to
prove Proposition \ref{ste.pro1.1}. In fact, by adding two key estimates \eqref{2.35} and \eqref{2.10} together, applying Lemma \ref{main.result4.3} and Lemma \ref{main.result4.4},
and taking $\delta>0$ sufficiently small, one concludes that
\begin{align}
 &(1+\beta t)^{\xi}\|(\varphi,\psi,\zeta)(t)\|_{\varepsilon,\beta,2}^{2}
+\beta^{3}\int_{0}^{t}(1+\beta
\tau)^{\xi}\|(\varphi,\psi,\zeta)(\tau)\|_{\varepsilon-3,\beta,2}^{2}d\tau\notag\\[2mm]
&\qquad\leq C\|(\varphi_{0},\psi_{0},\zeta_{0})\|_{\varepsilon,\beta,2}^{2}
+C\xi\beta\int_{0}^{t} (1+\beta
\tau)^{\xi-1}\|(\varphi,\psi,\zeta)(\tau)\|_{\varepsilon,\beta,2}^{2}d\tau.
\label{2.54}
\end{align}
In terms of \eqref{2.54}, applying an induction argument similar as \cite{SKawashima} and \cite{MNishikawa} with the choice of
  $\xi=(\lambda-\varepsilon)/3+\kappa$ for an arbitrary positive constant $\kappa$
and combining the elliptic estimates in Lemma \ref{main.result4} yield that
 \begin{align}
 &(1+\beta
 t)^{(\lambda-\varepsilon)/3+\kappa}(\|(\varphi,\psi,\zeta)(\tau)\|_{\varepsilon,\beta,2}^{2}+\|\sigma(\tau)\|_{\varepsilon,\beta,4}^{2})\notag\\[2mm]
 &
\quad+\beta^{3}\int_{0}^{t} (1+\beta
\tau)^{(\lambda-\varepsilon)/3+\kappa}
(\|(\varphi,\psi,\zeta)(\tau)\|_{\varepsilon-3,\beta,2}^{2}+\|\sigma(\tau)\|_{\varepsilon-3,\beta,4}^{2})d\tau
\notag\\[2mm]
&\leq C(1+\beta
 t)^{\kappa}\|(\varphi_{0},\psi_{0},\zeta_{0})\|_{\lambda,\beta,2}^{2},\notag
\end{align}
which proves the desired estimate \eqref{prop2.1r1} under the conditions \eqref{a}, \eqref{b}, \eqref{c} and \eqref{d}.
Then this completes the proof of Proposition \ref{ste.pro1.1}.
\end{proof}

\section{Energy estimates for the nondegenerate case}
In this section, we study the asymptotic stability of the
stationary solution to $\eqref{1.1}$ for the nondegenerate case
$\eqref{1.15a}$. As in the previous section, we shall only focus on the proof of the {\it a priori} estimates that will be given in Proposition \ref{ste.pro2}. The proof of the local-in-existence as well as the continuation argument is omitted for brevity.

\begin{proposition}\label{ste.pro2}
Let the same conditions on $T_{\infty}$, $u_{\infty}$ and
$\lambda$ as in Theorem \ref{1.2theorem} hold.

\medskip
(i) Let $(\varphi,\psi,\zeta,\sigma)$ be a solution to
$\eqref{1.16}$--$\eqref{1.19}$ which satisfies
$$
(e^{\lambda
x/2}\varphi,e^{\lambda x/2}\psi,e^{\lambda x/2}\zeta,e^{\lambda
x/2}\sigma)\in (\mathscr{X}_{2}([0,M]))^{3} \times
\mathscr{X}_{2}^2([0,M])
$$
for $M>0$. Then, there exist  constants $\delta>0$ and $C>0$ independent
of $M$ such that if the following conditions
\begin{center}
$\alpha>0$, $\beta\in(0,\lambda]$, and
$\beta+(|\phi_{b}|+\mathcal {N}_{\lambda}(M)+\alpha)/\beta \leq
\delta$
\end{center}
are satisfied, where
\begin{eqnarray*}
 \begin{aligned}[b]
 \mathcal {N}_{\lambda}(M):=\sup_{0\leq t\leq
M}\|(e^{\lambda x/2}\varphi,e^{\beta x/2}\psi,e^{\beta x/2}\zeta)(t)\|_{H^{2}},
\end{aligned}
\end{eqnarray*}
then it holds for any $t\in[0,M]$ that
\begin{align}
\|(e^{\beta x/2}\varphi,e^{\beta x/2}\psi,e^{\beta x/2}\zeta)(t)\|_{H^{2}}^{2}&+\|e^{\beta
x/2}\sigma(t)\|_{H^{4}}^{2}\notag\\[2mm]
&\leq C \|(e^{\lambda
x/2}\varphi_{0},e^{\lambda
x/2}\psi_{0},e^{\lambda
x/2}\zeta_{0})\|_{H^{2}}^{2}e^{-\alpha t}.\label{3.1}
\end{align}

\medskip
(ii) Let $(\varphi,\psi,\zeta,\sigma)$ be a solution to
$\eqref{1.16}$--$\eqref{1.19}$ over $[0,M]$ for $M>0$. Then, for any $\varepsilon\in(0,\lambda]$,
there exist constants $\delta>0$ and $C>0$ independent of $M$ such that if all the following conditions
\begin{equation*}
((1+\beta x)^{\lambda/2}\varphi,(1+\beta x)^{\lambda/2}\psi,(1+\beta
x)^{\lambda/2}\zeta,(1+\beta x)^{\lambda/2}\sigma)\in
(\mathscr{X}_{2}([0,M]))^{3} \times
\mathscr{X}_{2}^2([0,M])
\end{equation*}
and
\begin{eqnarray}\label{3.2b}
 \begin{aligned}[b]
\beta+(|\phi_{b}|+\mathcal {N}_{\lambda,\beta}(M))/\beta \leq
\delta,\quad \beta>0
\end{aligned}
\end{eqnarray}
are satisfied, then it holds for any $t\in[0,M]$ that
\begin{eqnarray}\label{3.3}
 \begin{aligned}[b]
\|(\varphi,\psi,\zeta)(t)\|_{\varepsilon,\beta,2}^{2}+\|\sigma(t)\|_{\varepsilon,\beta,4}^{2}\leq
C \|(\varphi_{0},\psi_{0},\zeta_{0})\|_{\lambda,\beta,2}^{2}(1+\beta
t)^{-(\lambda-\varepsilon)}.
\end{aligned}
\end{eqnarray}
\end{proposition}

 Since it is easier to treat the {\it a priori} estimate for the exponential weight than for the algebraic weight, we would only prove the second part of Proposition \ref{ste.pro2} in the case of algebraic weights for brevity. As in the previous section, we separate estimates into two parts that will be given in Lemma \ref{main.result2.987} and Lemma
 \ref{main.result2.987as}, respectively. Then,
 Proposition \ref{ste.pro2} is proved at the end of this section.

\begin{lemma}\label{main.result2.987}
Under the same conditions as in Proposition \ref{ste.pro2} (ii), for any $\varepsilon\in(0,\lambda]$, there exist constants $\delta>0$ and $C>0$ independent of $M$ such that it holds for any $t\in[0,M]$ and $\xi\geq 0$ that
\begin{eqnarray}
 &&(1+\beta t)^{\xi}\|(\varphi,\psi,\zeta)(t)\|_{\varepsilon,\beta,1}^{2}
+\beta\int_{0}^{t} (1+\beta
\tau)^{\xi} (\|(\varphi,\psi,\zeta)(\tau)\|_{\varepsilon-1,\beta,1}^{2}
+\|\sigma_{x}(\tau)\|_{\varepsilon-1,\beta}^{2})d\tau\notag\\[2mm]
&&\leq C\|(\varphi_{0},\psi_{0},\zeta_{0})\|_{\varepsilon,\beta,1}^{2}
+C\xi\beta\int_{0}^{t} (1+\beta
\tau)^{\xi-1}\|(\varphi,\psi,\zeta)(\tau)\|_{\varepsilon,\beta,1}^{2}d\tau.\label{3.9}
\end{eqnarray}
\end{lemma}

 \begin{proof}
 As in the proof of Lemma \ref{main.result2.9}, one can repeat the same procedure to obtain the identity \eqref{2.7}. It remains to re-estimate each term in $\eqref{2.7}$. First of all, one can still show \eqref{2.10a} in the same way so that  \eqref{2.11a} holds true. For the boundary terms on the left-hand side of \eqref{2.7}, the non-negativity estimates \eqref{2.25} and \eqref{2.11b} are also satisfied. Only the slight differences occur to estimates on $I_1$, $I_2$ and the right-hand term of $\eqref{2.7}$; it is indeed much easier to make estimates in the non-degenerate case than in the degenerate case considered before. In fact, for $I_1$, it holds that
 \begin{align}
 I_{1} &\geq\int_{\mathbb{R}_{+}}\varepsilon\beta
W_{\varepsilon-1,\beta}\Big\{
\frac{1}{2}(RT_{\infty}+1)|u_{\infty}|\varphi^{2}-RT_{\infty}\varphi\psi+\frac{1}{2}m|u_{\infty}|\psi^{2}\notag\\[2mm]
&\qquad\qquad\qquad\qquad+\frac{R
|u_{\infty}|}{2(\gamma-1)T_{\infty}}\zeta^{2}-R\zeta\psi+\sigma\psi\Big\}dx\notag\\[2mm]
&\quad
-C(\mathcal {N}_{\lambda,\beta}(M)+\phi_{b})\int_{\mathbb{R}_{+}}\varepsilon\beta
W_{\varepsilon-1,\beta}(\varphi^{2}+\psi^{2}+\zeta^{2})dx.
\label{3.12}
\end{align}
Furthermore, using the Cauchy-Schwarz inequality
$\sigma\psi\geq-(\frac{|u_{\infty}|}{2}\sigma^{2}+\frac{1}{2|u_{\infty}|}\psi^{2})$,
it follows from \eqref{3.12} that
\begin{eqnarray}\label{3.13}
 \begin{aligned}[b]I_{1} \geq&\int_{\mathbb{R}_{+}}\varepsilon\beta
W_{\varepsilon-1,\beta}\Big\{
\frac{|u_{\infty}|}{2}(RT_{\infty}+1)\varphi^{2}-RT_{\infty}\varphi\psi+\frac{m|u_{\infty}|^{2}-1}{2|u_{\infty}|}\psi^{2}\\[2mm]
&\qquad\qquad\qquad\qquad+\frac{R
|u_{\infty}|\zeta^{2}}{2(\gamma-1)T_{\infty}}-R\zeta\psi-\frac{|u_{\infty}|}{2}\sigma^{2}\Big\}dx\\[2mm] &
-C(\mathcal {N}_{\lambda,\beta}(M)+\phi_{b})\int_{\mathbb{R}_{+}}\varepsilon\beta
W_{\varepsilon-1,\beta}(\varphi^{2}+\psi^{2}+\zeta^{2})dx.
\end{aligned}
\end{eqnarray}
To deal with the bad term $-\int_{\mathbb{R}_{+}}\varepsilon\beta
W_{\varepsilon-1,\beta}\frac{|u_{\infty}|}{2}\sigma^{2} dx$ on the right-hand side of \eqref{3.13}, we first rewrite $\eqref{2.3}$ as the form of
\begin{eqnarray}\label{3.1phgl3}
 \begin{aligned}
\sigma_{xx}=\varphi+\sigma+(\tilde{n}-1)\varphi
+(e^{-\tilde{\phi}}-1)\sigma+\frac{\tilde{n}}{2}e^{\theta_{1}\varphi}\varphi^{2}-\frac{e^{-\tilde{\phi}}}{2}e^{-\theta_{2}\sigma}\sigma^{2},\
\ \  \theta_{1},\theta_{2}\in(0,1).
\end{aligned}
\end{eqnarray}
Then, by multiplying \eqref{3.1phgl3} by $-|u_{\infty}|\varepsilon\beta W_{\varepsilon-1,\beta}\sigma$, taking the integration over $\mathbb{R}_{+}$ and using the boundary condition $\sigma(t,0)=0$ as well as \eqref{1.14}, one has
\begin{align}
&|u_{\infty}|\varepsilon\beta\int_{\mathbb{R}_{+}}
W_{\varepsilon-1,\beta}\sigma_{x}^{2}dx+\int_{\mathbb{R}_{+}} |u_{\infty}|\varepsilon(\varepsilon-1)\beta^{2}W_{\varepsilon-2,\beta}\sigma\sigma_{x}dx
\notag\\[2mm]
&\leq -\int_{\mathbb{R}_{+}}|u_{\infty}|\varepsilon\beta W_{\varepsilon-1,\beta}\varphi \sigma dx
-\int_{\mathbb{R}_{+}}|u_{\infty}|\varepsilon\beta W_{\varepsilon-1,\beta}\sigma^{2}dx
\notag\\[2mm]
&\quad+C(\phi_{b}+\|\sigma\|_{\infty})\int_{\mathbb{R}_{+}}\varepsilon\beta
W_{\varepsilon-1,\beta}(\varphi^{2}+\sigma^{2})dx.
\label{3.1iup1}
\end{align}
Applying the Cauchy-Schwarz inequality $-\sigma\varphi\leq\frac{1}{2}\varphi^{2}+\frac{1}{2}\sigma^{2}$, the first two terms on the right-hand side of \eqref{3.1iup1} are bounded by
\begin{equation*}
\frac{1}{2}\int_{\mathbb{R}_{+}}|u_{\infty}|\varepsilon\beta W_{\varepsilon-1,\beta}\varphi^{2}dx
-\frac{1}{2}\int_{\mathbb{R}_{+}}|u_{\infty}|\varepsilon\beta W_{\varepsilon-1,\beta}\sigma^{2}dx.
\end{equation*}
To estimate the last term on the right-hand side of \eqref{3.1iup1}, from the Sobolev inequality, we notice that
\begin{equation*}
\|\sigma\|_{\infty}\leq C(\|\sigma\|+\|\sigma_{x}\|),
\end{equation*}
so it holds that $\|\sigma\|_{\infty}\leq C\|\varphi\|$ due to the elliptic estimate in Lemma \ref{main.result4}. Therefore, using $\|\sigma\|_{\varepsilon-1,\beta}\leq C\|\varphi\|_{\varepsilon-1,\beta}$ due to Lemma \ref{main.result4} once again, it follows that
\begin{equation*}
C(\phi_{b}+\|\sigma\|_{\infty})\int_{\mathbb{R}_{+}}\varepsilon\beta
W_{\varepsilon-1,\beta}(\varphi^{2}+\sigma^{2})dx
\leq C(\mathcal
{N}_{\lambda,\beta}(M)+\phi_{b})\int_{\mathbb{R}_{+}}\varepsilon\beta
W_{\varepsilon-1,\beta}\varphi^{2}dx.
\end{equation*}
The second term on the left hand side of \eqref{3.1iup1} can been treated by using the integration by parts together with the boundary condition $\sigma(t,0)=0$ as
 \begin{eqnarray*}
 \begin{aligned}
&\int_{\mathbb{R}_{+}} |u_{\infty}|\varepsilon(\varepsilon-1)\beta^{2}W_{\varepsilon-2,\beta}\sigma\sigma_{x}dx
=-\frac{1}{2}\int_{\mathbb{R}_{+}}|u_{\infty}|\varepsilon(\varepsilon-1)(\varepsilon-2)\beta^{3}W_{\varepsilon-3,\beta}\sigma^{2}dx,
\end{aligned}
\end{eqnarray*}
where by using  the elliptic estimate in Lemma \ref{main.result4}, it further holds that
\begin{equation*}
\left|\int_{\mathbb{R}_{+}}|u_{\infty}|\varepsilon(\varepsilon-1)(\varepsilon-2)\beta^{3}W_{\varepsilon-3,\beta}\sigma^{2}dx\right|\leq C\varepsilon\beta^3\int_{\mathbb{R}_{+}}W_{\varepsilon-1,\beta}\varphi^{2}dx.
\end{equation*}
Plugging all the above estimates into \eqref{3.1iup1} gives that
 \begin{align}
&\frac{1}{2}\int_{\mathbb{R}_{+}}|u_{\infty}|\varepsilon\beta W_{\varepsilon-1,\beta}\varphi^{2}dx
-\frac{1}{2}\int_{\mathbb{R}_{+}}|u_{\infty}|\varepsilon\beta W_{\varepsilon-1,\beta}\sigma^{2}dx\notag \\[2mm]
&\geq \int_{\mathbb{R}_{+}}|u_{\infty}|\varepsilon\beta
W_{\varepsilon-1,\beta}\sigma_{x}^{2}dx
-C(\mathcal
{N}_{\lambda,\beta}(M)+\phi_{b})\int_{\mathbb{R}_{+}}\varepsilon\beta
W_{\varepsilon-1,\beta}\varphi^{2}dx\notag\\[2mm]
&\quad-C\varepsilon\beta^3\int_{\mathbb{R}_{+}}
W_{\varepsilon-1,\beta}\varphi^{2}dx.
\label{3.1gfiuj}
\end{align}
We then substitute $\eqref{3.1gfiuj}$ back to $\eqref{3.13}$ and take $\delta>0$ suitably small so as to obtain
\begin{eqnarray}\label{3.13de}
 \begin{aligned}[b]I_{1} \geq&\int_{\mathbb{R}_{+}}\varepsilon\beta
W_{\varepsilon-1,\beta}\Big\{
\frac{|u_{\infty}|}{2}RT_{\infty}\varphi^{2}-RT_{\infty}\varphi\psi+\frac{m|u_{\infty}|^{2}-1}{2|u_{\infty}|}\psi^{2}+\frac{R
|u_{\infty}|\zeta^{2}}{2(\gamma-1)T_{\infty}}-R\zeta\psi\Big\}dx\\[2mm]
&+\int_{\mathbb{R}_{+}}|u_{\infty}|\varepsilon\beta
W_{\varepsilon-1,\beta}\sigma_{x}^{2}dx
-C(\mathcal {N}_{\lambda,\beta}(M)+\phi_{b})\int_{\mathbb{R}_{+}}\varepsilon\beta
W_{\varepsilon-1,\beta}(\varphi^{2}+\psi^{2}+\zeta^{2})dx\\[2mm]
&-C\beta\delta^{2}\int_{\mathbb{R}_{+}}
W_{\varepsilon-1,\beta}\varphi^{2}dx,
\end{aligned}
\end{eqnarray}
where  $\frac{\gamma RT_{\infty}+1}{m}<
u^{2}_{\infty}$ and $\eqref{3.2b}$
 are applied in the last inequality.

Now we estimate $I_{2}$ and the last term  in
$\eqref{2.7}$. In fact, it holds that
 \begin{align}
&|I_{2}|+\left|\int_{\mathbb{R}_{+}} W_{\varepsilon,\beta}\CN_{1} dx\right|\notag\\[2mm]
&\leq  C(\mathcal
{N}_{\lambda,\beta}(M)+\phi_{b})\int_{\mathbb{R}_{+}}\varepsilon\beta
W_{\varepsilon-1,\beta}(\varphi^{2}+\psi^{2}+\zeta^{2}+\varphi_{x}^{2}+\psi_{x}^{2}+\zeta_{x}^{2})dx\notag\\[2mm]
&\leq
C\beta\delta\|(\varphi,\psi,\zeta)\|_{\varepsilon-1,\beta,1}^{2},\label{3.14}
\end{align}
where we have used Lemma \ref{1.1} (iii), $\eqref{1.16}$, $\eqref{3.2b}$,
$\lambda\geq 2$, the Cauchy-Schwarz inequality and the elliptic estimate
in Lemma \ref{main.result4}.

Substituting $\eqref{3.13de}$ and $\eqref{3.14}$
into $\eqref{2.7}$, we have
\begin{equation}
\frac{d}{dt}\int_{\mathbb{R}_{+}} W_{\varepsilon,\beta}(e^{-\tilde{\phi}}\CE_0
+\CE_1^{\rm x}+\frac{1}{2}\tilde{n}^{2}\varphi^{2})dx
+\beta\|(\varphi,\psi,\zeta)\|_{\varepsilon-1,\beta,1}^{2}+\beta\|\sigma_{x}\|_{\varepsilon-1,\beta}^{2}
\leq 0,\label{3.15}
\end{equation}
provided that $\delta>0$ is sufficiently small, where $\CE_0$ and $\CE_1^{\rm x}$ are defined in \eqref{def.e0} and \eqref{def.e1x} respectively. Therefore, the desired estimate \eqref{3.9} follows from multiplying $\eqref{3.15}$ by $(1+\beta \tau)^{\xi}$ and integrating the resulting inequality over $(0,t)$. This then completes the proof of Lemma \ref{main.result2.987}.
\end{proof}

\begin{lemma}\label{main.result2.987as}
Under the same conditions as in Proposition \ref{ste.pro2}, for any
$\varepsilon\in(0,\lambda]$, there exist positive constants $C$ and
$\delta$ independent of $M$ such that it holds for any
$t\in[0,M]$ and $\xi\geq 0$ that
 \begin{align}
 &(1+\beta t)^{\xi}\|(\varphi_{t},\psi_{t},\zeta_{t})\|_{\varepsilon,\beta,1}^{2}
+\beta\int_{0}^{t} (1+\beta
\tau)^{\xi}\left(\|(\varphi_{t},\psi_{t},\zeta_{t})(\tau)\|_{\varepsilon-1,\beta,1}^{2}
+\|\sigma_{tx}(\tau)\|_{\varepsilon-1,\beta}^{2}\right)d\tau\notag\\[2mm]
&\leq C\|(\varphi_{t0},\psi_{t0},\zeta_{t0})\|_{\varepsilon,\beta,1}^{2}
+C\xi\beta\int_{0}^{t} (1+\beta
\tau)^{\xi-1}\|(\varphi_{t},\psi_{t},\zeta_{t})(\tau)\|_{\varepsilon,\beta,1}^{2}d\tau\notag\\[2mm]
&\quad+C\delta\beta\int_{0}^{t}
(1+\beta
\tau)^{\xi}\|(\varphi,\psi,\zeta)(\tau)\|_{\varepsilon-1,\beta,2}^{2}d\tau.\label{3.16}
\end{align}
\end{lemma}

\begin{proof}
After taking one more time derivative, \eqref{3.16} follows similarly as for obtaining \eqref{3.9} in Lemma \ref{main.result2.987} and hence details of the proof are omitted for brevity.
\end{proof}

Now, following Lemma \ref{main.result2.987} and Lemma \ref{main.result2.987as} above, we are ready to give the

\begin{proof}[Proof of Proposition \ref{ste.pro2}]
Adding the obtained estimates $\eqref{3.16}$ and $\eqref{3.9}$ together, applying Lemma \ref{main.result4.3} and Lemma \ref{main.result4.4},
and taking $\delta>0$ sufficiently small, we conclude that
 \begin{align}
 &(1+\beta t)^{\xi}\|(\varphi,\psi,\zeta)(\tau)\|_{\varepsilon,\beta,2}^{2}
+\beta\int_{0}^{t} (1+\beta
\tau)^{\xi}\|(\varphi,\psi,\zeta)(\tau)\|_{\varepsilon-1,\beta,2}^{2}d\tau
\notag \\[2mm]
&\leq
 C\|(\varphi_{0},\psi_{0},\zeta_{0})\|_{\varepsilon,\beta,2}^{2}
+C\xi\beta\int_{0}^{t} (1+\beta
\tau)^{\xi-1}\|(\varphi,\psi,\zeta)(\tau)\|_{\varepsilon,\beta,2}^{2}d\tau,\label{3.17}
\end{align}
for any $t\geq 0$ and $\xi\geq 0$.
Then, in terms of \eqref{3.17}, employing the induction argument similar as \cite{SKawashima} and \cite{MNishikawa} with $\xi=\lambda-\varepsilon+\kappa$  for an arbitrary positive constant $\kappa$
and combining the elliptic estimates in Lemma \ref{main.result4} yield that for any $t\geq 0$,
 \begin{align*}
 &(1+\beta
 t)^{\lambda-\varepsilon+\kappa}(\|(\varphi,\psi,\zeta)(\tau)\|_{\varepsilon,\beta,2}^{2}+\|\sigma(\tau)\|_{\varepsilon,\beta,4}^{2})\notag \\[2mm]
 &
\quad+\beta\int_{0}^{t} (1+\beta
\tau)^{\lambda-\varepsilon+\kappa}(\|(\varphi,\psi,\zeta)(\tau)\|_{\varepsilon-1,\beta,2}^{2}+\|\sigma(\tau)\|_{\varepsilon-1,\beta,4}^{2})d\tau
\notag \\[2mm]
&\leq C(1+\beta
 t)^{\kappa}\|(\varphi_{0},\psi_{0},\zeta_{0})\|_{\lambda,\beta,2}^{2},
\end{align*}
which proves \eqref{3.3}. This then completes the proof of the second part (ii)  of Proposition \ref{ste.pro2}. As mentioned before, for the part (i) corresponding to the exponential weight case, the proof of \eqref{3.1} follows in a similar way and thus is omitted for brevity. We therefore conclude the proof of Proposition \ref{ste.pro2}.
\end{proof}

\section{Appendix}\label{Append}
In this appendix, we will give some basic results used in the proof of Proposition \ref{ste.pro1.1} and Proposition \ref{ste.pro2}. Those lemmas below are similar to ones obtained in \cite{NOS}.

\begin{lemma}\label{main.result4}
Consider the elliptic equation $\eqref{1.17}$.

\medskip
(i) Let  $((1+\mu x)^{\lambda/2}\varphi,(1+\mu
x)^{\lambda/2}\sigma)\in \mathscr{X}_{2}([0,M]) \times
\mathscr{X}_{2}^{2}([0,M])$ for positive constants
$\lambda$ and $\mu$. Then, for any constant $c_0\in(0,2],$
there exist
 positive constants $\delta$ and $C$ independent of $M$ such that if all the conidtions
$\alpha\leq\lambda$, $\beta\in(0,\mu]$, $|\alpha\beta|\leq c_0$ and
$|\phi_{b}|+\mathcal {N}_{\lambda,\beta}(M)\leq \delta$ are satisfied, then $\sigma$ satisfies $(1+\beta x)^{\alpha/2}\sigma\in
 \mathscr{X}_{2}^{2}([0,M])$ with
\begin{eqnarray}\label{4.1}
 \begin{aligned}[b]
 \|(1+\beta x)^{\alpha/2}\partial_{t}^{i}\sigma\|_{H^{j}}\leq
 C\|(1+\beta x)^{\alpha/2}\varphi\|_{H^{i+j-2}},\ \ \ i\in
 \mathbb{Z}\cap[0,2], \ \  j\in\mathbb{Z}\cap[2,4-i].
\end{aligned}
\end{eqnarray}

\medskip
(ii) Let $(e^{\lambda x/2}\varphi,e^{\lambda x/2}\sigma)\in
\mathscr{X}_{2}([0,M]) \times \mathscr{X}_{2}^{2}([0,M])$ for  positive constants $\lambda$. Then, for any constant
$c_0\in(0,\sqrt{2}]$, there exist
 positive constants $\delta$ and $C$ independent of $M$ such that if the conditions $\beta\in(0,c_0]$ and $|\phi_{b}|+\mathcal
{N}_{\lambda}(M)\leq \delta$ are satisfied, then $\sigma$ satisfies
$e^{\beta x/2}\sigma\in
 \mathscr{X}_{2}^{2}([0,M])$ with
\begin{eqnarray}\label{4.2}
 \begin{aligned}[b]
 \|e^{\beta x/2}\partial_{t}^{i}\sigma\|_{H^{j}}\leq
 C\|e^{\beta x/2}\varphi\|_{H^{i+j-2}},\ \ \ i\in
 \mathbb{Z}\cap[0,2],\ \  j\in\mathbb{Z}\cap[2,4-i].
\end{aligned}
\end{eqnarray}
\end{lemma}

\begin{proof}
The desired estimates \eqref{4.1} and \eqref{4.2} can be derived by the standard elliptic estimate on \eqref{1.17}. For brevity, we omit their proofs.
\end{proof}

\begin{lemma}\label{main.result4.2}
Under the same assumptions as in either Proposition \ref{ste.pro1.1} for the
degenerate case or Proposition \ref{ste.pro2} (ii) for the nondegenerate case,
it holds for any $t\in[0,M]$ and $\alpha \leq\lambda/2$ that
\begin{eqnarray}
&\dis \|((1+\beta x)^{\alpha}{(\varphi,\psi,\zeta)},(1+\beta x)^{\alpha}{(\varphi_{x},\psi_{x},\zeta_{x})})(t)\|_{L^{\infty}(\mathbb{R}^{+})}\leq
 C\mathcal {N}_{\lambda,\beta}(M),\label{4.3}\\[2mm]
&\dis \|(1+\beta x)^{\alpha}{(\varphi_{t},\psi_{t},\zeta_{t})}(t)\|_{L^{\infty}(\mathbb{R}^{+})}\leq
 C\mathcal {N}_{\lambda,\beta}(M).\label{4.4}
\end{eqnarray}
\end{lemma}

\begin{proof}
The first estimate \eqref{4.3} directly follows from the Sobolev inequality that
\begin{eqnarray*}
&&\|((1+\beta x)^{\alpha}{(\varphi,\psi,\zeta)},(1+\beta
x)^{\alpha}{(\varphi_{x},\psi_{x},\zeta_{x})})(t)\|_{L^{\infty}(\mathbb{R}^{+})}\\[2mm]
&&\leq  C\|((1+\beta x)^{\alpha}{(\varphi,\psi,\zeta)},(1+\beta
x)^{\alpha}{(\varphi_{x},\psi_{x},\zeta_{x})})(t)\|_{H^{1}(\mathbb{R}^{+})}\\[2mm]
&&\leq  C\|((1+\beta x)^{\alpha}{(\varphi,\psi,\zeta)}(t)\|_{H^{2}(\mathbb{R}^{+})}.
\end{eqnarray*}
The second estimate $\eqref{4.4}$ immediately follows from $\eqref{4.3}$
owing to $\eqref{1.16}$, $\eqref{1.17}$, Lemma \ref{1.1} and Lemma
\ref{main.result4} (i).
\end{proof}

\begin{lemma}\label{main.result4.3}
For the nondegenerate case, we assume the same conditions as in
Proposition \ref{ste.pro2} (ii) and let $\delta$ be suitably small. Then it holds
for $\xi=\varepsilon-1$ or $\varepsilon$ that
\begin{eqnarray}\label{4.5}
 \begin{aligned}[b]
 \|\partial_{t}^{i}{(\varphi,\psi,\zeta)}\|_{\xi,\beta,j}\leq
 C\|{(\varphi,\psi,\zeta)}\|_{\xi,\beta,i+j},
\end{aligned}
\end{eqnarray}
where $(i,j)\in\{(i,j)\in\mathbb{Z}^{2}|i,j\geq 0, i+j\leq 2\}$. For the
degenerate case, we assume the same conditions as in Proposition \ref{ste.pro1.1}
and let $\delta$ be suitably small. Then it holds for
$\xi=\varepsilon-3$, $\varepsilon-1$ or $\varepsilon$ that
\begin{eqnarray}
&\dis \|{(\varphi_{t},\psi_{t},\zeta_{t})}\|_{\xi,\beta} \leq
C\|({\varphi_{x},\psi_{x},\zeta_{x}},\sigma_{x})\|_{\xi,\beta}+C\beta\|{(\varphi,\psi,\zeta)}\|_{\xi-2,\beta}
 \leq C\|{(\varphi,\psi,\zeta)}\|_{\xi,\beta,1},\label{4.6}\\[2mm]
&\dis \|({\varphi_{tx},\psi_{tx},\zeta_{tx}},{\varphi_{tt},\psi_{tt},\zeta_{tt}})\|_{\xi,\beta}
 \leq C\|{(\varphi,\psi,\zeta)}\|_{\xi,\beta,2}.\label{4.7}
\end{eqnarray}
\end{lemma}

\begin{proof}
Those estimates \eqref{4.5}, \eqref{4.6} and \eqref{4.7} can be derived by the governing system
 $\eqref{1.16}$ and $\eqref{1.17}$ as well as the time-derivative system with the help of Lemma
\ref{main.result4} (i).
\end{proof}

\begin{lemma}\label{main.result4.4}

For the nondegenerate case, we assume the same conditions as in
Proposition \ref{ste.pro2} (ii) and let $\delta$ be suitably small. Then it holds
for $\xi=\varepsilon-1$ or $\varepsilon$ that
\begin{eqnarray*}
&\dis \|{(\varphi_{x},\psi_{x},\zeta_{x})}\|_{\xi,\beta}\leq
 C\|({\varphi_{t},\psi_{t},\zeta_{t}},{\varphi,\psi,\zeta})\|_{\xi,\beta},\\[2mm]
&\dis \|{(\varphi_{xx},\psi_{xx},\zeta_{xx})}\|_{\xi,\beta}\leq
 C\|({\varphi_{tt},\psi_{tt},\zeta_{tt}},{\varphi_{t},\psi_{t},\zeta_{t}},{\varphi,\psi,\zeta})\|_{\xi,\beta}.
\end{eqnarray*}
For the degenerate case, we assume the same conditions as in
Proposition \ref{ste.pro1.1} and let $\delta$ be suitably small. Then it holds for
$\xi=\varepsilon-3$, $\varepsilon-1$, or $\varepsilon$ that
\begin{eqnarray*}
\|{(\varphi_{x},\psi_{x},\zeta_{x})}\|_{\xi,\beta}\leq
C\|({\varphi_{t},\psi_{t},\zeta_{t}},\sigma_{x})\|_{\xi,\beta}+C\beta\|({\varphi,\psi,\zeta})\|_{\xi-2,\beta}
 \leq C\|({\varphi_{t},\psi_{t},\zeta_{t}},{\varphi,\psi,\zeta})\|_{\xi,\beta},
\end{eqnarray*}
and
\begin{equation*}
\begin{aligned}
\|{(\varphi_{xx},\psi_{xx},\zeta_{xx})}\|_{\xi,\beta}&\leq
 C\|({\varphi_{tt},\psi_{tt},\zeta_{tt}},{\varphi_{x},\psi_{x},\zeta_{x}},{\varphi,\psi,\zeta})\|_{\xi,\beta}
 \\[2mm]
 &\leq C\|({\varphi_{tt},\psi_{tt},\zeta_{tt}},{\varphi_{t},\psi_{t},\zeta_{t}},{\varphi,\psi,\zeta})\|_{\xi,\beta}.
\end{aligned}
\end{equation*}
\end{lemma}

\begin{proof} By $\eqref{1.16}$, we have
  \begin{eqnarray}\label{4.7a}
\begin{aligned} & A_{1} \left(\begin{array} {c}
\varphi \\
\psi\\
\zeta\\
\end{array} \right)_{t}+A_{2}
\left(\begin{array} {c}
\varphi \\
\psi\\
\zeta\\
\end{array} \right)_{x}
=-A_{3}\left(\begin{array} {c}
\tilde{v} \\
\tilde{u}\\
\tilde{T}\\
\end{array} \right)_{x}+\left(\begin{array} {c}
0 \\
\sigma_{x}\\
0\\
\end{array} \right),
\end{aligned}
\end{eqnarray}
where  $A_{1}$, $A_{2}$ and $A_{3}$ are corresponding matrices appearing in  $\eqref{1.16}$. Recall
$$
u=\psi+(\tilde{u}-u_{\infty})+u_{\infty}, \quad
T=\zeta+(\tilde{T}-T_{\infty})+T_{\infty},\ \text{ and }
u^{2}_{\infty}\geq\frac{\gamma RT_{\infty}+1}{m}.
$$
One then sees that $A_{2}$ is
regular and each entry of $A_{2}^{-1}$ is bounded under the condition that $\mathcal
{N}_{\lambda,\beta}(M)+|\phi_{b}|\ll 1.$ Hence, whenever it holds that $\mathcal
{N}_{\lambda,\beta}(M)+|\phi_{b}|\ll 1$,  one can solve $\eqref{4.7a}$
for $(\varphi_{x},\psi_{x},\zeta_{x})^{T}$ so as to obtain
\begin{eqnarray}\label{4.8}
\begin{aligned} &  \left(\begin{array} {c}
\varphi \\
\psi\\
\zeta\\
\end{array} \right)_{x}
 =-M_{2}^{-1}M_{1}\left(\begin{array} {c}
\varphi \\
\psi\\
\zeta\\
\end{array} \right)_{t}
-M_{2}^{-1}M_{3}\left(\begin{array} {c}
\tilde{v} \\
\tilde{u}\\
\tilde{T}\\
\end{array} \right)_{x}+M_{2}^{-1}\left(\begin{array} {c}
0 \\
\sigma_{x}\\
0\\
\end{array} \right).
\end{aligned}
\end{eqnarray}
Under the same conditions as in Proposition \ref{ste.pro2} (ii), it follows from $\eqref{4.8}$ that
\begin{eqnarray*}
\|{(\varphi_{x},\psi_{x},\zeta_{x})}\|_{\xi,\beta}\leq
C\|({\varphi_{t},\psi_{t},\zeta_{t}},{\varphi,\psi,\zeta},\sigma_{x})\|_{\xi,\beta} \leq
C\|({\varphi_{t},\psi_{t},\zeta_{t}},{\varphi,\psi,\zeta})\|_{\xi,\beta}.
\end{eqnarray*}
Similarly, under the same conditions as in Proposition \ref{ste.pro1.1}, \eqref{4.8} also gives that
\begin{equation*}
\begin{aligned}
\|{(\varphi_{x},\psi_{x},\zeta_{x})}\|_{\xi,\beta} &\leq
C\|({\varphi_{t},\psi_{t},\zeta_{t}},\sigma_{x})\|_{\xi,\beta}+C\beta\|({\varphi,\psi,\zeta})\|_{\xi-2,\beta}
  \\[2mm]
   &\leq C\|({\varphi_{t},\psi_{t},\zeta_{t}},{\varphi,\psi,\zeta})\|_{\xi,\beta}.
\end{aligned}
\end{equation*}
The other desired estimates for the higher order derivatives can be shown in the same way as above. This ends the proof of Lemma \ref{main.result4.4}.
\end{proof}


\noindent{\bf Acknowledgements:} Renjun Duan was supported by the General Research Fund (Project No. 14301719) from RGC of Hong Kong. Haiyan Yin was supported by the National Natural Science Foundation of China (Grant Nos. 12071163
and 11601165) and the Promotion Program
for Young and Middle-aged Teacher in Science and Technology Research of Huaqiao
University (Grant No. ZQN-PY602).
Changjiang Zhu was supported by the National Natural Science Foundation of China (Grant Nos. 11771150, 11831003 and 11926346)
and Guangdong Basic and Applied Basic Research Foundation (Grant No. 2020B1515310015).


\end{document}